\documentclass[11pt,a4paper, envcountsame,mathserif]{amsart}
\usepackage[usenames,dvipsnames]{color}

\usepackage[utf8]{inputenc}

 \usepackage[colorlinks,citecolor=blue,urlcolor=blue, linkcolor=blue, backref]{hyperref}

 \usepackage{amsmath, amssymb, xspace}
 \usepackage{graphicx}

 \usepackage[all]{xy}

\usepackage{tikz}
\usetikzlibrary{arrows,snakes,positioning,backgrounds,shadows}

\usepackage{undertilde}

\newtheorem{theorem}{Theorem}[section]

\newtheorem{thm}[theorem]{Theorem}

\newtheorem{fact}[theorem]{Fact}

\newtheorem{example}[theorem]{Example}

\newtheorem{proposition}[theorem]{Proposition}

\newtheorem{prop}[theorem]{Proposition}

\newtheorem{claim}[theorem]{Claim}

\newtheorem{conjecture}[theorem]{Conjecture}

\newtheorem{lemma}[theorem]{Lemma}

\newtheorem{question}[theorem]{Question}

\theoremstyle{definition}
\newtheorem{definition}[theorem]{Definition}

\newcommand{\NN}{{\mathbb{N}}}
\newcommand{\RR}{{\mathbb{R}}}

\newcommand{\QQ}{{\mathbb{Q}}}
\newcommand{\ZZ}{{\mathbb{Z}}}
\newcommand{\sub}{\subseteq}
\newcommand{\sN}[1]{_{#1\in \NN}}
\newcommand{\uhr}[1]{\! \upharpoonright_{#1}}
\newcommand{\ML}{Martin-L{\"o}f}

\newcommand{\PI}[1]{\Pi^0_{#1}}
\newcommand{\PPI}{\PI{1}}

\newcommand{\bi}{\begin{itemize}}
\newcommand{\ei}{\end{itemize}}
\newcommand{\bc}{\begin{center}}
\newcommand{\ec}{\end{center}}

\newcommand{\ES}{\emptyset}

\newcommand{\tp}[1]{2^{#1}}

\newcommand{\fa}{\forall}

\newcommand{\la}{\langle}
\newcommand{\ra}{\rangle}

\newcommand{\leT}{\le_{\mathrm{T}}}

\newcommand{\MLR}{\mbox{\rm \textsf{MLR}}}

\newcommand{\n}{\noindent}
\newcommand{\constr}{\n {\it Construction.}\ \ }
\newcommand{\vsps}{\vspace{3pt}}

\newcommand{\vsp}{\vspace{6pt}}

\newcommand{\w}{\omega}
\newcommand{\s}{\sigma}

\newcommand{\rest}[1]{\! \upharpoonright_{#1}} 
\newcommand{\wck}{\omega_1^{ck}}

\newcommand{\lland}{\, \land \, }

\newcommand \seq[1]{{\left\langle{#1}\right\rangle}}

\newcommand\+[1]{\mathcal{#1}}

\newcommand{\wt}{\widetilde}

\newcommand{\RA}{\Rightarrow}

\newcommand{\range}{\ensuremath{\mathrm{range}}}

  \DeclareMathOperator{\High}{High}

  \newcommand{\CR}{\mbox{\rm \textsf{CR}}}

\DeclareMathOperator{\KP}{\textit{K}\,}
\DeclareMathOperator{\KS}{\textit{C}\,}
\DeclareMathOperator{\KPt}{\textit{K}}

\DeclareMathOperator{\m}{\mathbf{m}}

\newcommand{\leK}{\mathrel{\le_{\textrm{LK}}}}
\newcommand{\leMLR}{\mathrel{\le_{\textrm{LR}}}}
\newcommand{\cnd}{\,|\,}

\begin{document}

\title{Logic Blog 2014}

 \author{Editor: Andr\'e Nies}

\maketitle


 {
The logic blog is for
\bi \item Rapid announcement of new results in logic
\item putting up results and their proofs for further research
\item archiving results that fit nowhere else
\item sandbox for later publication in a journal.   \ei

Some time after the year has ended, the blog is finalised and  posted on arXiv.
\vsp
\begin{tabbing} 
   \href{http://arxiv.org/abs/1403.5719}{Logic Blog 2013} \ \ \ \   \= (Link: \texttt{http://arxiv.org/abs/1403.5719})  \\

    \href{http://arxiv.org/abs/1302.3686}{Logic Blog 2012}  \> (Link: \texttt{http://arxiv.org/abs/1302.3686})   \\

 \href{http://arxiv.org/abs/1403.5721}{Logic Blog 2011}   \> (Link: \texttt{http://arxiv.org/abs/1403.5721})   \\

 \href{http://dx.doi.org/2292/9821}{Logic Blog 2010}   \> (Link: \texttt{http://dx.doi.org/2292/9821})  
     \end{tabbing}

\vsp

\n {\bf How does the blog work?}

\vsp

\n {\bf Writing and editing.}  The source files are in a shared dropbox.
 Ask \email{andre@cs.auckland.ac.nz}    to gain access.

\vsp

\n {\bf Citing.}  Postings can be cited.  An example of a citation is:

\vsp

\n  H.\ Towsner, \emph{Computability of Ergodic Convergence}. In  Andr\'e Nies (editor),  Logic Blog, 2012, Part 1, Section 1, available at
\url{http://arxiv.org/abs/1302.3686}.}

\vsp

\n {\bf Announcements on the wordpress front end.}  The Logic Blog has a \href{http://logicblogfrontend.hoelzl.fr/
}{front end}  managed by Rupert H\"olzl.   

\n (Link: \texttt{http://logicblogfrontend.hoelzl.fr/})

\vsps
When you post source on the logic blog in the dropbox, you can post a comment on the front end alerting the community, and possibly summarising the result in brief.  The front end is also good for posting questions. It allows MathJax.
 
%
\newpage
\tableofcontents

\part{Randomness via algorithmic tests}

\section{CCR 2014: Open questions}
Written by Rupert H\"olzl and Andr\'e Nies in June 2014.

\subsection{Relativising computable randomness}
Let $\CR$ denote the class of computably random subsets of $\NN$; these are the  sets on which no computable  martingale succeeds.  Let $\MLR$ denote the class of \ML\  random sets.  
\begin{question}[van Lambalgen's theorem for $\CR$]
Let $X$ be computably random, and let  $Y$ be computably random relative to $X$. Is $X \oplus Y$ computably random?
\end{question}

\begin{question}[Joseph Miller]
Let $X$ be computably random. Is there an incomputable oracle $A$ such that $X$ is computably random relative to $A$?
\end{question}
Let $\High(\CR, \MLR)$ denote the class of oracles $A$ such that $\CR^A \sub \MLR$. Every PA complete set is in $\High(\CR, \MLR)$ by \cite[Lemma 5.1]{Hirschfeldt.Nies.ea:07} and \cite{FSY11}.
\begin{question}[Joseph Miller, \cite{FSY11}, implicit in \cite{Hirschfeldt.Nies.ea:07}] Is every set in  $\High(\CR, \MLR)$ PA-complete?
 \end{question}

We write $A \le_{CR} B$ if $\CR^B \sub \CR^A$. Clearly $\leT$ implies $\le_{\CR}$.
\begin{question}[Nies] Is  $\le_{CR}$ equivalent to $\leT$?
\end{question}
Miyabe has provided an affirmative answer to an analogous question for Schnorr randomness.


\subsection{Lowness for Martin-Löf randomness}

We discuss two superclasses of the low for ML-random sets, and ask whether they are strictly larger.

The first  can be seen as an analog in computability theory of the notion of null-additivity in set theory (see \cite[Ch.\ 2]{Bartoszynski.Judah:book}).  Let $A \triangle B$ denote symmetric difference of sets $A,B$.
\begin{definition}[Kihara]
We say that $A\sub \NN $ is \emph{ML-additive} if $Z \in \MLR$ implies $Z \triangle A\in \MLR$ for each $Z$. 
\end{definition}

Clearly, if a set is  low for ML-randomness then it  is ML-additive.
\begin{question}
Is every ML-additive oracle low for ML-randomness?
\end{question}
We could replace the map $X \to X \triangle A$ by a general  $A$-computable measure preserving bijection $\Phi$ on Cantor Space. We say $A$ is \emph{$\Phi$-preserving} if  $Z \in \MLR$ implies $\Phi(Z) \in \MLR$ for each $Z$. If  $A$ is low for ML-randomness, then it is $\Phi$-preserving for each such $\Phi$. It is open whether the converse holds.  

Kjos-Hanssen's intuition is that a sufficiently sparse $A$ might be ML-additive. In this connection, whether, to the contrary,   each ML-random  $Z$ is derandomized by arbitrarily fast growing functions: for each function $g$, is there $f$ dominating $g$ such that $Z - \range (f)$ is not ML-random?

Kihara and Miyabe have   recent  results characterising   additivity for Schnorr randomness.  For instance, they show this coincides with lowness for uniform Schnorr randomness.

We say that $A$ is low   for a ML-random set $R$ if $R$ is   ML-random relative to  $A$. The second class strengthens the property of being  low for $R$ . 
\begin{definition}[Yu]
Let $A\sub \NN$ and let   $R$ be ML-random.  We say that   $A$ is \emph{absolutely low for}   $R$ if for all sets $Y$,
\[R \in \MLR^Y \rightarrow R\in\MLR^{Y\oplus A}.\]
\end{definition}

Every low for ML-random is low for $\Omega$.
 
\begin{fact}[Stephan and Yu] Let  $A$ be $K$-trivial. Then $A$ is absolutely low for $\Omega$. 
\end{fact} 
\begin{proof} Suppose $Y$ is low for $\Omega$. By the low for $\Omega$ basis theorem relative to $Y$, there is $C \ge_T Y$ such that $C$ is low for $\Omega$ and also $C$ is PA-complete relative to $Y$. By a recent result of Simpson and Stephan, this implies that $A \leT C$. Therefore $A \oplus Y$ is low for $\Omega$. 
\end{proof} 
We ask whether  the converse holds.
\begin{question}[Yu]
Is every set that is absolutely low for $\Omega$   $K$-trivial?
\end{question}

\begin{definition}[Merkle,Yu]
We say an oracle $A$ is super-absolutely low for $\Omega$  if for each oracle  $Y$ which is low for $\Omega$,
\[\forall n(K^Y(n)=K^{A\oplus Y}(n)+O(1)).\]
\end{definition}

Note that every such set  is  absolutely   low for $\Omega$ by definition. In fact it  must be  low for $K$, by letting $Y= \ES$. We ask the converse.

\begin{question}[Merkle,Yu]
Is every low for  $K$ set super-absolutely low for~$\Omega$? 
\end{question}

\n {\it Note (Apr. 2015) The previous two questions have now been solved in the affirmative by Greenberg, Miller, Monin, and Turetsky in a paper called ``Two more characterizations of the $K$-trivials", to appear in NDJFL.}

\subsection{Descriptive string complexity}

We say that $A \sub \NN$ is \emph{infinitely often $K$-trivial} if $\exists^\infty n \, K(X \upharpoonright n) \le^+ K(n)$.   For a reference see \cite{Barmpalias.Vlek:11}, who show for instance that every truth-table degree contains a set of this kind. Bauwens has recently shown that i.o.\ $K$-trivial and i.o.\ $C$-trivial are incomparable notions.

\begin{question}[Barmpalias] Suppose that $\forall r \exists n > r \, K(X \upharpoonright n) \le^+ K(r,n)$. Is $X$ infinitely often $K$-trivial? \end{question}

\section{Turetsky: Partial martingales for computable measures}

By Daniel Turetsky, March 2014.

Recall that martingales are only required to be defined on cylinders of positive measure.

\begin{definition}
If $\mu$ is a probability measure on $2^\w$, a {\em martingale with respect to $\mu$} is a partial function $M: 2^{<\w} \to [0, \infty)$ satisfying:
\begin{itemize}
\item For all $\sigma \in 2^{<\w}$, $M(\sigma* 0)\mu(\sigma* 0) + M(\sigma* 1)\mu(\sigma* 1) = M(\sigma)\mu(\sigma)$.
\item For all $\sigma \in 2^{<\w}$, if $\mu([\sigma]) > 0$, then $M(\sigma)$ is defined.
\end{itemize}
\end{definition}

Jason Rute  \cite{Rute:14}  asked if we need to allow martingales to be partial like this for computable randomness.  The answer is yes.

\begin{theorem}
There is a computable probability measure $\mu$ and an $X \in 2^{\w}$ which is not computably $\mu$-random, yet no total computable $\mu$-martingale succeeds on $X$.
\end{theorem}

\begin{proof}
We build $\mu$ and a computable $\mu$-martingale $N$.  The real $X$ will then be a sort of ``true path'' through $N$.  Let $\seq{M_i}_{i \in \w}$ be an enumeration of all potential $\mu$-martingales.  That is, these are partial computable maps $M_i: 2^{<\w} \times \w \to \QQ_{\geq 0}$ such that for every $\sigma \in 2^{<\w}$, the sequence $\seq{M_i(\sigma, n)}_{n \in \w}$ is a (partial) Cauchy name for a non-negative real.  We will abuse notation and refer to $M_i$ as a partial map from $2^{<\w}$ to $\RR_{\geq 0}$.

During the construction, we will choose various positive rationals $\epsilon(\tau)$ for $\tau \in 2^{<\w}$.  We will define
\[
M_\tau = \sum_{\tau(i) = 1} \epsilon(\tau\uhr{i}) M_i.
\]
Thus each $M_\tau$ is a partial computable function, and if $\tau_0 \subseteq \tau_1$, $M_{\tau_0} \leq M_{\tau_1}$.  Note that $M_{\seq{}}$ is identically 0, and in particular is total computable.

If $Y = \{ i \in \w : \text{$M_i$ is a total computable $\mu$-martingale}\}$, our intention is to ensure that $\liminf_{n\to\infty} M_\tau(X\uhr{n}) < \infty$ for all $\tau \prec Y$.  By the savings trick, this will suffice to show that no total computable $\mu$-martingale succeeds on $X$.

\

{\em Strategy for $\tau$:}

Each $\tau \in 2^{<\w}$ will have a {\em coding location} $\sigma_\tau \in 2^{<\w}$, and will have a {\em preferred bit} $b_\tau \in \{0, 1\}$.  We will threaten to make $\mu(\sigma_\tau* b_\tau) = 0$.  The important point is that while we are threatening, we are under no obligation to define $N$ along extensions of $\sigma_\tau* b_\tau$.  Nor must we make much in the way of definitions for $\mu$ along proper extensions of $\sigma_\tau* b_\tau$.  Of course, we must uniformly give a Cauchy name for $\mu(\sigma)$ for all $\sigma$, but we can do this above $\sigma_\tau* b_\tau$ without imposing any actual restraint.

We explain what is meant by the above.  Our threatening to make $\mu(\sigma_\tau* b_\tau) = 0$ takes the form of declaring $\mu(\sigma_\tau* b_\tau) \in [0, 2^{-s}]$ at stage $s$.  Note that this fulfills all of our obligations to make a Cauchy name.  We can simultaneously declare $\mu(\rho) \in [0, 2^{-s}]$ for all $\rho \supset \sigma_\tau* b_\tau$, and this also meets our obligations to make Cauchy names.  However, this second declaration carries no new information; it is entirely implied by the fact that $\mu(\sigma_\tau* b_\tau) \in [0, 2^{-s}]$.  If at some stage $s$ we choose to stop threatening and instead declare $\mu(\sigma_\tau* b_\tau) = 2^{-s}$, our previous declarations for the $\mu(\rho)$ impose no restraint; we are as free to define $\mu(\rho)$ as we would be if we had made no declarations for it at all.  In particular, any $\rho \supset \sigma_\tau* b_\tau$ and any $b$ are available as our next coding location once we decide to stop threatening at $\sigma_\tau* b_\tau$.  So we will choose coding locations where $M_\tau$ has not gained much capital.

We now describe the strategy for $\tau$, given $\sigma_\tau$ and $b_\tau$.  Our strategy will be based on several inductive assumptions:
\begin{itemize}
\item $\mu(\sigma_\tau)$ has already been defined.
\item $N(\sigma_\tau* b_\tau)$ has already been defined.
\item For all $i < |\tau|-1$ with $\tau(i) = 1$, $\epsilon(\tau\uhr{i})$ is already defined.  Thus $M_{\tau^-}$ is already defined.  If $\tau(|\tau|-1) = 1$, $M_\tau$ is not yet defined because we lack $\epsilon(\tau^-)$, but this is the only missing ingredient, while if $\tau(|\tau|-1) = 0$, $M_\tau$ is already defined.
\item $M_{\tau^-}(\sigma_\tau* b_\tau) < 2$.
\end{itemize}

The strategy is as follows:
\begin{enumerate}
\item Begin threatening $\mu(\sigma_\tau * b_\tau) = 0$.  That is, at stage $s$, for every $\rho \supseteq \sigma_\tau* b_\tau$, declare $\mu(\rho) \in [0, 2^{-s}]$.  Also declare $\mu(\sigma_\tau * (1 - b_\tau)) \in [\mu(\sigma_\tau) - 2^{-s}, \mu(\sigma_\tau)]$.  The extensions $\rho \supset \sigma_\tau*(1 - b_\tau)$ are irrelevant, so make declarations for $\mu(\rho)$ consistent with the uniform distribution above $\sigma_\tau*(1-b_\tau)$.  The strategy moves on to the next step, but this process continues until stated otherwise.
\item If $\tau(|\tau|-1) = 0$, skip to Step~\ref{skip_to_step}.
\item Wait for $M_{|\tau|-1}$ to make a declaration that $M_{|\tau|-1}(\sigma_\tau* b_\tau) < a$ for some positive rational $a$.
\item\label{epsilon_define_step} Define $\epsilon(\tau^-) = \frac{2^{-2|\tau|}}{a}$.
\item\label{skip_to_step} Wait for $M_\tau$ to be defined on all extensions of $\sigma_\tau* b_\tau$ of length $|\sigma_\tau|+1 + 2|\tau| + 2$.
\item Cease threatening $\mu(\sigma_\tau * b_\tau) = 0$, and instead declare $\mu(\sigma_\tau * b_\tau) = 2^{-s}$, while $\mu(\sigma_\tau* (1 - b_\tau)) = \mu(\sigma_\tau) - 2^{-s}$.  Continue to extend $\mu$ above $\sigma_\tau* (1 - b_\tau)$ uniformly (these strings are still irrelevant).
\item\label{found_strings_step} If there are two strings $\sigma_0$ and $\sigma_1$ both extending $\sigma_\tau* b_\tau$ and of length $|\sigma_\tau|+1 + 2|\tau|+2$, with $M_\tau(\sigma_j) < M_\tau(\sigma_\tau * b_\tau) + 2^{-2|\tau|}$ for $j \in \{0, 1\}$ and $\sigma_0^- \neq \sigma_1^-$, then do the following:
\begin{enumerate}
\item Declare $\sigma_{\tau* j} = \sigma_j^-$ and $b_{\tau* j} = \sigma_j(|\sigma_j|-1)$.
\item For all $\rho$ extending $\sigma_\tau* b_\tau$ with $\rho \not \supset \sigma_j^-$ for either $j$, define $\mu(\rho)$ uniformly.
\item Define $N(\sigma_j^-) = N(\sigma_j^-* 0) = N(\sigma_j^-* 1) = 2^{2|\tau|}N(\sigma_\tau* b_\tau)$ for both $j$, and define $N(\rho) = 0$ for all $\rho$ extending $\sigma_\tau* b_\tau$ but incomparable with both $\sigma_j^-$.  For all $\rho$ extending $\sigma_j^-* (1 - b_{\tau* j})$, define $N(\rho) = N(\sigma_j^-)$.
\item Begin the strategies for $\tau* 0$ and $\tau* 1$, and end the current strategy.
\end{enumerate}
\item\label{failure_step} If there are no such $\sigma_0$ and $\sigma_1$, define $\mu(\rho)$ and $N(\rho)$ uniformly on all extensions of $\sigma_\tau* b_\tau$ and end the current strategy.
\end{enumerate}

\

{\em Construction:}

Begin by letting $\sigma_{\seq{0}} = \seq{0}$, $b_\seq{0} = 0$, $\sigma_{\seq{1}} = \seq{1}$ and $b_\seq{1} = 0$.  Define $\mu(\seq{}) = 1$, $\mu(\seq{0}) = \mu(\seq{1}) = 1/2$.  Define $N(\rho) = 1$ for $|\rho| < 3$, and $N(\rho) = 1$ for $\rho$ with $\rho(1) = 1$.  Begin the strategies for $\seq{0}$ and $\seq{1}$.

\

{\em Verification:}

\begin{claim}
When we begin the strategy for $\tau$, our inductive assumptions are met.
\end{claim}

\begin{proof}
This is easily seen to hold for the base case $|\tau| = 1$.

An examination of Step \ref{found_strings_step} shows that $\mu(\sigma_\tau)$ and $N(\sigma_\tau* b_\tau)$ will both be defined before the strategy for $\tau$ is begun.  Inductively, by Step \ref{epsilon_define_step}, we know that $\epsilon(\tau\uhr{i})$ is defined for all $i < |\tau|-1$ with $\tau(i) = 1$.  It remains only to show that $M_{\tau^-}(\sigma_\tau* b_\tau) < 2$.

We argue by induction that $M_{\tau^-}(\sigma_\tau* b_\tau) < 2 - 2^{-|\tau^-|}$.  Since $M_\seq{}$ is identically 0, clearly this holds for $|\tau| = 1$.

If $\tau(|\tau|-1) = 0$, then
\begin{eqnarray*}
M_\tau(\sigma_\tau* b_\tau) &=& M_{\tau^-}(\sigma_\tau* b_\tau)\\
&<& 2 - 2^{-|\tau^-|}\\
&<& 2 - 2^{-|\tau|} - 2^{-2|\tau|},
\end{eqnarray*}
where the first inequality is by the inductive hypothesis.

If instead $\tau(|\tau|-1) = 1$, then by our action at Step \ref{epsilon_define_step}, we have that
\begin{eqnarray*}
M_{\tau}(\sigma_\tau* b_\tau) &=& M_{\tau^-}(\sigma_\tau* b_\tau) + \epsilon(\tau^-) M_{|\tau|-1}(\sigma_\tau* b_\tau)\\
&<& M_{\tau^-}(\sigma_\tau* b_\tau) + 2^{-2|\tau|}\\
&<& 2 - 2^{-|\tau^-|} + 2^{-2|\tau|}\\
&\leq& 2 - 2^{-|\tau|} - 2^{-2|\tau|},
\end{eqnarray*}
where the second inequality is by the inductive hypothesis.

If we begin the strategies for $\tau* 0$ and $\tau* 1$, then by our action at Step \ref{found_strings_step}, we know that $M_\tau(\sigma_{\tau* j}* b_{\tau* j}) < M_\tau(\sigma_\tau* b_\tau) + 2^{-2|\tau|}$, which completes the proof.
\end{proof}

\begin{claim}
$\mu$ is a total computable measure.
\end{claim}

\begin{proof}
By examination.  If we threaten forever at $\mu(\sigma_\tau* b_\tau)$, then our action for $\tau$ ensures that $\mu$ is defined on all extensions of $\sigma_\tau$.  If we stop threatening, then our action for $\tau$ ensures that $\mu$ is defined on all $\rho$ extending $\sigma_\tau$ with $\rho$ not extending one of $\sigma_{\tau* 0}$ or $\sigma_{\tau* 1}$.  We also begin the strategies for $\tau* 0$ and $\tau* 1$, who ensure that $\mu$ is defined on the extensions of $\sigma_{\tau*0}$ or $\sigma_{\tau*1}$, respectively.
\end{proof}

\begin{claim}
$N$ is a computable $\mu$-martingale.
\end{claim}

\begin{proof}
By examination.  If we threaten forever at $\mu(\sigma_\tau* b_\tau)$, then we are not required to converge on extensions of $\sigma_\tau * b_\tau$.  If we stop threatening and find $\sigma_0$ and $\sigma_1$, then our action at Step \ref{found_strings_step} defines $N$ for all $\rho$ extending $\sigma_\tau* b_\tau$ except those extending $\sigma_j$, and then we begin the strategies for $\tau* j$, who will take responsibility for $N$ above $\sigma_j$.  Since we extended $\mu$ uniformly, the definition given precisely satisfies the martingale condition.

If we stop threatening but do not find $\sigma_0$ and $\sigma_1$, then our action at Step \ref{failure_step} defines $N$ for all $\rho$ extending $\sigma_\tau* b_\tau$.  Since the extension is uniform, the definition satisfies the martingale condition.
\end{proof}

As before, let $Y = \{ i \in \w : \text{$M_i$ is a total computable $\mu$-martingale}\}$.
\begin{claim}
If $\tau \prec Y$, then we eventually begin the strategy for $\tau$.
\end{claim}

\begin{proof}
By induction on $|\tau|$.  Since $\tau \prec Y$, $M_\tau$ is a total computable $\mu$-martingale.

Clearly we eventually stop threatening at $\mu(\sigma_\tau* b_\tau)$.  It suffices to show that the desired $\sigma_0$ and $\sigma_1$ exist.  But suppose not.  Then for all but at most 1 string $\sigma$ of length $|\sigma_\tau| + 1 + 2|\tau| + 1$,
\[
M_\tau(\sigma) \geq M_\tau(\sigma_\tau* b_\tau) + 2^{-2|\tau|}.
\]
Since in this case we define $\mu$ uniformly by Step \ref{failure_step}, we have
\begin{eqnarray*}
\sum_{\sigma \in 2^{|\sigma_\tau|+1 + 2|\tau|+1}} \mu(\sigma) M_\tau(\sigma) &=& \sum_\sigma 2^{-(2|\tau|+1)}\mu(\sigma_\tau* b_\tau) M_\tau(\sigma)\\
&\geq& (2^{2|\tau|+1}-1) \mu(\sigma_\tau* b_\tau) (M_\tau(\sigma_\tau* b_\tau) + 2^{-2|\tau|})\\
&=& \mu(\sigma_\tau* b_\tau) M_\tau(\sigma_\tau* b_\tau) (2^{2|\tau|+1}-1) \left(1 + \frac{2^{-2|\tau|}}{M_\tau(\sigma_\tau* b_\tau)}\right).
\end{eqnarray*}
Recalling that $M_\tau(\sigma_\tau* b_\tau) < 2$, consider
\begin{eqnarray*}
(2^{2|\tau|+1}-1)(1 + 2^{-2|\tau|-1}) &=& 2^{2|\tau|+1} + 1 - 1 - 2^{-2|\tau|-1}\\
&=& 2^{2|\tau|+1} - 2^{-2|\tau|-1}\\
&\geq& 2^1 - 2^{-1}\\
&>& 1.
\end{eqnarray*}
But this contradicts the martingale condition.
\end{proof}

Now, let $X = \bigcup_{\tau \prec Y} \sigma_\tau$.  Since $\liminf M_\tau(X\uhr{n}) \leq 2$ for all $\tau \prec Y$, we see that no total computable $\mu$-martingale succeeds on $X$.  On the other hand, by construction we see that $N(\sigma_\tau) \geq 2^{2|\tau^-|}$ for all $\tau \prec Y$, and so $N$ succeeds on $Y$.
\end{proof}

\section{Kjos-Hanssen: Answer to a question of Brendle et al. 2014}
\subsection{The first question, Question 4 (6)}
	Here is a solution to Brendle, Brooke-Taylor, Ng and Nies \cite{Brendle.Brooke.ea:14}, Question 4 (6).
	Kjos-Hanssen and Stephan worked on this on June 16, 2014 in Singapore.
	$A$ is  weakly meager engulfing if it computes a meager set that contains all recursive reals. This is equivalent to being (high or DNR).
	\begin{thm}
		There is weakly meager engulfing set that does not compute a Schnorr random and is of hyperimmune degree.
	\end{thm}

	We use a modification of Ted Slaman's construction from the warm-up section of the paper ``Comparing DNR and WWKL'' \cite{Kjos.Lempp.ea:04}.
	There, one finds a construction of a function $G$ which is DNR (hence weakly meager engulfing), recursive in $0'$ (hence hyperimmune), and computes no ML-random.
	To show it computes no Schnorr random it thus suffices to modify the construction so that $G$ is not high. We do this by in fact making $G$ low.

	Ted Slaman mentioned (private communication in 2003 or 2004) that he thought the construction could be made low, but did not tell the proof. It turns out  to be fairly easy:

	Recall that $G$ is a union of strings $G_s$. At stages $2e+1$ we will take care of the requirements as in the original construction (making sure $G$ computes no ML-random).
	At stages $s=2e$ we work on the lowness requirements. We ask, is there a good or ``bushy'' tree starting from $G_s$ consisting of strings making $e$ in the jump of $G$?
	If so, take an element of such a tree that's also acceptable (which will always exist) and use it as $G_{s+1}$.
	If not, add to the list of acceptability requirements that there should not exist such a tree starting from any future $G_t$.
	This way we decide at say stage $2e$ whether $e$ is in the jump of $G$, hence in the end $G$ is low.  The acceptability also ensures that $G$ is DNR.

\section{Higuchi: Injective Images of Computable Randomness}

(By Kojiro Higuchi at Chiba University, Japan.) We give a short and self-contained  proof of a result that can be seen  as a  special case of Rute~\cite[Thm.\ 7.9]{Rute:14}.   Let us use $\mu$ to denote the fair coin measure on the Cantor space $\{0,1\}^\omega$. 
This paper gives a theorem such that the image of a computable random real under computable injection $\Phi$ over the Cantor space $\{0,1\}^\omega$ is computable random with respect to the measure $\mu\circ\Phi^{-1}$.   

\begin{definition}
Let $\Phi:\{0,1\}^\omega\to\{0,1\}^\omega$ be a computable function. 
A function $b:\{0,1\}^{<\omega}\to\mathbb{R}_{\ge 0}$ is called {\it $\Phi$-martingale} if 
$b$ satisfies the equation 
$$b(\tau)\mu(\Phi^{-1}(\tau\{0,1\}^\omega))=b(\tau 0)\mu(\Phi^{-1}(\tau 0\{0,1\}^\omega))+b(\tau 1)\mu(\Phi^{-1}(\tau1\{0,1\}^\omega))$$
for all $\tau\in\{0,1\}^{<\omega}$. 
We define 
$${\rm Succ}(b)=\{Y\in\{0,1\}^\omega:\limsup_{n\to\infty}b(Y\upharpoonright n)=\infty\}.$$ 
We say that $b$ {\it succeeds} on $Y$ if $Y\in{\rm Succ}(b)$. 
\end{definition}

\begin{definition}
Let $b:\{0,1\}^{<\omega}\to\mathbb{R}_{\ge 0}$ be a $\Phi$-martingale for a computable function $\Phi:\{0,1\}^\omega\to\{0,1\}^\omega$. 
We say that $b$ satisfies the {\it saving property} if 
for all $k\in\omega$ and all $Y\in{\rm Succ}(b)$, 
there exists $\tau\subsetneq Y$ such that 
$b(\rho)\ge k$ for any $\rho\in\tau\{0,1\}^{<\omega}$. 
\end{definition}

We do not give a proof of the following lemma which tells us that we can always modify a given $\Phi$-computable martingale to make it have the saving property without changing its succeeding set. 
A proof can be given just as a proof for the usual computable martingale. 

\begin{lemma}\label{L1}
For any computable $\Phi:\{0,1\}^\omega\to\{0,1\}^\omega$ and any computable $\Phi$-martingale $b:\{0,1\}^{<\omega}\to\mathbb{R}_{\ge 0}$, 
there exists a computable $\Phi$-martingale $b^\prime:\{0,1\}^{<\omega}\to\mathbb{R}_{\ge 0}$ with the saving property such that 
${\rm Succ}(b)={\rm Succ}(b^\prime)$. 
\qed
\end{lemma}

\begin{lemma}\label{L2}
Let $\Phi:\{0,1\}^\omega\to\{0,1\}^\omega$ be a computable function and let $b:\{0,1\}^{<\omega}\to\mathbb{R}_{\ge 0}$ be a $\Phi$-martingale. 
Suppose that the limit 
$$b^\prime(\sigma)=\lim_{n\to\infty}\sum_{\tau\in\{0,1\}^n}b(\tau)\dfrac{\mu(\sigma\{0,1\}^\omega\cap\Phi^{-1}(\tau\{0,1\}^\omega))}{\mu(\sigma\{0,1\}^\omega)}$$
exists for all $\sigma\in\{0,1\}^{<\omega}$. 
Then $b^\prime$ is a martingale. 
Furthermore, 
$$\Phi^{-1}({\rm Succ}(b))\subset{\rm Succ}(b^\prime)$$ 
holds if $b$ satisfies the saving property. 
\end{lemma}
\begin{proof}
By assumption, $b^\prime$ is a function from $\{0,1\}^{<\omega}$ into $\mathbb{R}_{\ge 0}$. 
Thus to show that $b^\prime$ is a martingale, we need to prove the equality for the martingales. 
Fix $\sigma\in\{0,1\}^{<\omega}$, 
We have the following equalities 
\begin{align*}
&b^\prime(\sigma 0)+b^\prime(\sigma 1)\\
=&\lim_{n\to\infty}\sum_{\tau\in\{0,1\}^n}b(\tau)\dfrac{\mu(\sigma 0\{0,1\}^\omega\cap\Phi^{-1}(\tau\{0,1\}^\omega))+\mu(\sigma 1\{0,1\}^\omega\cap\Phi^{-1}(\tau\{0,1\}^\omega))}{2^{-1}\mu(\sigma\{0,1\}^\omega)}\\
=&2\lim_{n\to\infty}\sum_{\tau\in\{0,1\}^n}b(\tau)\dfrac{\mu(\sigma\{0,1\}^\omega\cap\Phi^{-1}(\tau\{0,1\}^\omega))}{\mu(\sigma\{0,1\}^\omega)}\\
=&2b^\prime(\sigma). 
\end{align*}
Thus, $b^\prime$ is a martingale. 

Assume that $b$ satisfies the saving property. 
Fix $X\in\{0,1\}^\omega$ such that $\Phi(X)\in{\rm Succ}(b)$ and fix $k\in\omega$. 
By the saving property of $b$, 
we can find a finite initial segment $\tau$ of $\Phi(X)$ such that 
$b(\rho)\ge k$ for all $\rho\in\tau\{0,1\}^{<\omega}$. 
By the continuity of $\Phi$, 
there exists a finite initial segment $\sigma$ of $X$ with $\Phi(\sigma\{0,1\}^\omega)\subset\tau\{0,1\}^\omega$. 
Now, we have the following inequalities 
\begin{align*}
&b^\prime(\sigma)=\lim_{n\to\infty}\sum_{\rho\in\{0,1\}^n}b(\rho)\dfrac{\mu(\sigma\{0,1\}^\omega\cap\Phi^{-1}(\rho\{0,1\}^\omega))}{\mu(\sigma\{0,1\}^\omega)}\\
=&\lim_{n\to\infty}\sum\left\{b(\rho)\dfrac{\mu(\sigma\{0,1\}^\omega\cap\Phi^{-1}(\rho\{0,1\}^\omega))}{\mu(\sigma\{0,1\}^\omega)}:\rho\in\{0,1\}^n, \rho\supset\tau\right\}\\
\ge&\lim_{n\to\infty}\sum\left\{k\dfrac{\mu(\sigma\{0,1\}^\omega\cap\Phi^{-1}(\rho\{0,1\}^\omega))}{\mu(\sigma\{0,1\}^\omega)}:\rho\in\{0,1\}^n, \rho\supset\tau\right\}=k.
\end{align*}
Since we fixed $k$ arbitrarily, $\lim_{n\to\infty}b^\prime(X\upharpoonright n)=\infty$. 
Thus $\Phi^{-1}({\rm Succ}(b))\subset{\rm Succ}(b^\prime)$ holds. 
\end{proof}

\begin{lemma}\label{L3}
Let $\Phi:\{0,1\}^\omega\to\{0,1\}^\omega$ be a computable injection 
and let $b:\{0,1\}^{<\omega}\to\mathbb{R}_{\ge 0}$ be a computable $\Phi$-martingale. 
Then the function $b^\prime$ defined by 
$$b^\prime(\sigma)=\lim_{n\to\infty}\sum_{\tau\in\{0,1\}^n}b(\tau)\dfrac{\mu(\sigma\{0,1\}^\omega\cap\Phi^{-1}(\tau\{0,1\}^\omega))}{\mu(\sigma\{0,1\}^\omega)}$$
is a (total) computable martingale. 
\end{lemma}
\begin{proof}
Note that $\sigma\{0,1\}^\omega$ and $\{0,1\}^\omega\setminus\sigma\{0,1\}^\omega$ are both compact. 
Thus their images under $\Phi$ are again compact (and therefore closed) by the continuity of $\Phi$, and they are disjoint by the injectivity of $\Phi$. 
Thus there exists a natural number $n\in\omega$ such that 
the sets $\{\Phi(X)\upharpoonright n: X\in\sigma\{0,1\}^\omega\}$ and $\{\Phi(X)\upharpoonright n: X\in\{0,1\}^\omega\setminus\sigma\{0,1\}^\omega\}$ are disjoint. 
Note that such a natural number $n$ can be found uniformly in a given $\sigma\in\{0,1\}^{<\omega}$ 
since $\Phi$ is computable. 
If a natural $n$ satisfies the property 
\begin{equation}\label{DP}
\{\Phi(X)\upharpoonright n: X\supset\sigma\}\cap\{\Phi(X)\upharpoonright n: X\not\supset\sigma\}=\emptyset,
\end{equation}
then so does $n+1$. 
Moreover, it $n$ satisfies \eqref{DP}, then the equalities 
\begin{align*}
&\sum_{\tau\in\{0,1\}^n}b(\tau)\dfrac{\mu(\sigma\{0,1\}^\omega\cap\Phi^{-1}(\tau\{0,1\}^\omega))}{\mu(\sigma\{0,1\}^\omega)}\\
=&\sum\left\{b(\tau)\dfrac{\mu(\Phi^{-1}(\tau\{0,1\}^\omega))}{\mu(\sigma\{0,1\}^\omega)}:\exists X\supset\sigma, \tau=\Phi(X)\upharpoonright n\right\}\\
=&\sum\left\{b(\tau i)\dfrac{\mu(\Phi^{-1}(\tau i\{0,1\}^\omega))}{\mu(\sigma\{0,1\}^\omega)}:i\in\{0,1\}, \exists X\supset\sigma, \tau=\Phi(X)\upharpoonright n\right\}\\ 
=&\sum_{\tau\in\{0,1\}^{n+1}}b(\tau)\dfrac{\mu(\sigma\{0,1\}^\omega\cap\Phi^{-1}(\tau\{0,1\}^\omega))}{\mu(\sigma\{0,1\}^\omega)}
\end{align*}
holds since $b$ is a $\Phi$-martingale. 
It follows that 
$$b^\prime(\sigma)=\sum_{\tau\in\{0,1\}^n}b(\tau)\dfrac{\mu(\sigma\{0,1\}^\omega\cap\Phi^{-1}(\tau\{0,1\}^\omega))}{\mu(\sigma\{0,1\}^\omega)}.$$
Since we can find a natural number $n$ satisfying the above equation uniformly in a given $\sigma$ as we noted, 
$b^\prime$ is a (total) computable function. 
Together with the previous lemma, we conclude that $b^\prime$ is a computable martingale. 
\end{proof}

\begin{theorem}
Let $\Phi:\{0,1\}^\omega\to\{0,1\}^\omega$ be a computable injection. 
Then the images of computable random sequences under $\Phi$ are computable random w.r.t. the measure $\mu\circ\Phi^{-1}$. 
\end{theorem}
\begin{proof}
Fix $X\in\{0,1\}^\omega$ whose image under $\Phi$ is not computable random w.r.t. $\mu\circ\Phi^{-1}$. 
Choose a computable $\Phi$-martingale $b$ succeeding on $\Phi(X)$. 
By Lemma \ref{L1}, we may safely assume that $b$ satisfies the saving property. 
Define $b^\prime$ as 
$$b^\prime(\sigma)=\lim_{n\to\infty}\sum_{\tau\in\{0,1\}^n}b(\tau)\dfrac{\mu(\sigma\{0,1\}^\omega\cap\Phi^{-1}(\tau\{0,1\}^\omega))}{\mu(\sigma\{0,1\}^\omega)}$$
for all $\sigma\in\{0,1\}^{<\omega}$. 
By Lemma \ref{L2}, we know that $b^\prime$ is a computable martingale. 
Finally, we have ${\rm Succ}(b^\prime)\supset\Phi^{-1}({\rm Succ}(b))$ by Lemma \ref{L2}. 
Thus $b^\prime$ succeeds on $X$ since $b$ does on $\Phi(X)$. 
It follows that $X$ is not computable random. 
This completes the proof. 
\end{proof}

\section[Nies: Normality of a real relative to  non-integral bases]{Nies: Normality of a real relative to  non-integral bases, and uniform distribution}

By Andr\'e Nies, started March-April 2013. Represents work with Figueira.
\newcommand{\fract}{\mbox{\rm \textsf{frac}}}
\subsection{Normality}
For a real $x$ let $\fract (x)$ denote the fractional part $x-\lfloor x \rfloor$. (Some sources use $\{x\}$, which may  stem from the time of typewriters when fonts and symbols were limited.) A sequence $\seq {y_j}\sN j$ of reals in $[0,1]$ is \emph{uniformly distributed} if for each interval $[u,v]\sub [0,1]$, the proportion  of $i< N$ with  $y_j \in [u,v]$ tends to $v-u$ as $N \to \infty$. In  other words, the Bernoulli($(v-u)$) distributed random variables $1_{[u,v]}(y_j)$ satisfy the law of large numbers.
\begin{definition}\label{def:normal other bases} Let $r> 1$. We say that $x \in \RR$ is \emph{normal in base $r$} if the sequence $\seq {\fract {(xr^n)}}\sN n$ is uniformly distributed in $[0,1]$.
\end{definition}
For every $r>1$, the set of reals that are  normal for base  $r$ is conull. However, it is unknown whether, for instance, $1/2$ is normal in base $3/2$.

A real $x$ is  \emph{absolutely normal}  if it is normal in all integer bases $>1$. Let us say that~$x$ is \emph{rationally normal} if it is normal in all rational  bases $>1$.

In the following we will see that the rationally normal reals form a proper subset of the  absolutely normal reals, even though they  still are conull. This is a special case of a result by Brown, Moran and Pearce  \cite[Thm 2]{Brown.Moran.etal:86}.

Sets $A,B \sub (1, \infty)$ are called \emph{multiplicatively independent} (m.i.) if there are no $a\in A, b \in B, r,s \in \NN$ such that $a^r= b^s$. For instance, $A= \NN\setminus \{0,1\} $ and $B=\{3/2\}$ are m.i. The result of Brown et al.\ says that given m.i.\ sets of algebraic numbers, every real is the sum of four numbers that are normal for all bases in $A$, but none in $B$. In particular, there are uncountably many reals that are absolutely normal, but not normal  for the base $3/2$.
Figueira and Nies \cite{Figueira.Nies:13}  show that every polynomially random real is absolutely normal.  This is also in upcoming work of  Mayordomo and Hitchcock, who use   polynomial dimension as an intermediate notion. The following likely conjecture would extend that  result.
\begin{conjecture} Every polynomially random real is rationally normal.
\end{conjecture}

\subsection{Uniform distribution and Weyl's theorem}

Weyl  (1916) proved that for every sequence $\seq {a_i}\sN i$ of distinct integers, for almost every $y\in [0,1]$,  the sequence $\seq {\fract{(a_i y)}}\sN i$  is uniformly distributed. Avigad (Uniform distribution and algorithmic randomness, 2012) calls $y$ {\em UD random} if this happens for every \emph{computable} sequence of distinct integers. He shows that Schnorr randomness implies UD randomess; on the other hand, UD does not imply even Kurtz, so it is not an actual randomness notion.

\begin{conjecture} \label{Conj:UD poly} Suppose the function $1^j \to a_j$ is polynomial time computable. Then for every polynomially random real $y$, the sequence $\seq {\fract{(a_i y)}}\sN i$  is uniformly distributed.
\end{conjecture}

Weyl's theorem can be extended to certain sequences $\seq {u_i}$ of distinct reals: if there is $k >0$ such that $|u_i-u_j| > k$ for $i \neq j$, then $\seq {\fract{(u_i y)}}\sN i$  is uniformly distributed for almost every $y$. See for instance \cite[Cor 4.7]{Kuipers.Niederreiter:12} (first edition 1976),  where it is derived from a more general theorem due to Koksma.

Joseph Miller, Figueira and Nies working at  Miramar near Buenos Aires,  have given an alternative  proof,  reducing the more general case to the case of integers. This proof is effective. In particular, for any computable sequence of reals as above, Schnorr randomness is sufficient. A feasible version of their  proof should extend an affirmative answer to  Conjecture~\ref{Conj:UD poly} to the situation where $a_j$ is a polynomial time computable real uniformly in $1^j $.

We turn this Weyl's theorem extension into a notion of randomness:
\begin{definition}
$x$ is {\em strongly UD-random} iff for every computable sequence $(u_n)$ of rationals such that $(\exists k>0)(\forall i,j)\ [i\not=j\to|u_i-u_j|>k]$, the sequence $(u_nx)$ is uniformly distributed modulo one.
\end{definition}

Strong UD-random implies rational normality. Is strong UD-randomness a strictly stronger notion than UD-randomness? 

Check connection with \cite{Franklin.Greenberg.etal:12,Bienvenu.Day.etal:12}, but observe that the transformation $T(x)$ defined as $xq$ modulo one (for some fixed rational $q$) is not measure preserving, i.e.\ it is not true that $\lambda A=\lambda T^{-1}(A)$ for all measurable set $A$. For $q=3/2$ and $A=[0,1/2]$ we have $T^{-1}(A)=[0,1/3]\cup [2/3,1]$, and so $1/2=\lambda A \not=\lambda T^{-1}(A)=2/3$.

Let $(x_n)$ be any sequence of reals.
For $h\in\ZZ$ and $n\in\NN$ define
$$
S_h(n,x)=\frac{1}{n}\sum_{j<n}e^{2\pi i h x_j}.
$$
Weyl's criterion states that the sequence $(x_n)$ is uniformly distributed modulo one if and only if $\lim_{n\to\infty}S_h(n,x)=0$ for all integers $h\not=0$.

%

\bigskip

Observe that
\begin{eqnarray*}
|S_h(n,x)|^2 &=& \frac{1}{n^2}\left(\sum_{j<n}e^{2\pi i h x_n}\right)\left(\sum_{j<n}e^{-2\pi i h x_n}\right)
\\
&=&\frac{1}{n^2}\left(n+\sum_{\stackrel{i,j< n}{i\not=j}}e^{2\pi i h (x_i-x_j)}\right)
\\
&=&\frac{1}{n^2}\left(n+\sum_{i<j<n}e^{2\pi i h (x_i-x_j)}+e^{2\pi i h (x_j-x_i)}\right)
\\
&=&\frac{1}{n^2}\left(n+2\sum_{i<j<n}\cos(2\pi h (x_i-x_j))\right),
\end{eqnarray*}
and so $|S_h(n,x)|^2$ is uniformly computable in $h$ and $n$.

\bigskip

Recall that Schnorr randomness can be characterized in terms of total Solovay tests \cite{Downey.Griffiths:04}. A {\em total Solovay test} is a sequence $(A_n)$ of uniformly $\Sigma^0_1$ classes such that $\sum_n\lambda A_n$ is finite and a computable real. A set {\em passes} this test if it is in only finitely many of the $A_n$. A set is Schnorr random iff it passes all total Solovay tests.

We modify the proof of Avigad that Schnorr randomness implies UD-randommess to obtain the following:

\begin{theorem}
Schnorr randomness implies strong UD-randomness.
\end{theorem}

\begin{proof}
Suppose $(u_n)$ is a sequence of rationals such that $(\forall i,j)\ [i\not=j\to|u_i-u_j|>k]$ for some $k>0$ and that the sequence $(u_nx)$ is not uniformly distributed modulo one. Let $x_i=u_ix$. By the Weyl criterion, there is $h\in\ZZ\setminus\{0\}$ such that it is not true that
$
\lim_{n\to \infty} S_h(n,x)=0.
$
Since for any $m$ such that $n^2 \leq m < (n+1)^2$ we have
$|S(n,x)|\leq |S(m^2,x)|+2/\sqrt{n}$, we conclude that it is not true that
$
\lim_{n\to \infty} S_h(n^2,x)=0.
$
Define
$$
A_n=\{z\in[0,1] \mid |S_h(n^2,z)|^2>1/\log n\}.
$$
Suppose that $\delta>0$ is such that there are infinitely many $n$ such that $|S_h(n^2,x)|^2>\delta$. For large enough $n$ this implies that $|S_h(n^2,x)|^2>1/\log n$, and so $x\in A_n$ for infinitely many $n$, so it only remains to show that $(A_n)$ is a total Solovay test.

By Markov's inequality, we have
$$
\lambda A_n\leq \log n \int_0^1|S_h(n^2,x)|^2dx,
$$
and by the proof of Koksma's General Metric Theorem in \cite[Theorem 4.3]{Kuipers.Niederreiter:12},
$$
\int_0^1|S_h(n,x)|^2dx\leq \frac{1}{n}+\frac{8\log (3n)}{|h|kn}.
$$
Hence
$$
\sum_n\lambda A_n\leq \sum_n \log n \left(\frac{1}{n^2}+\frac{8\log (3n^2)}{|h|kn^2}\right)<\infty.
$$

It is clear that $(A_n)$ is a sequence of uniformly $\Sigma^0_1$ classes, that $\lambda A_n$ is computable uniformly in $n$, and hence that $\sum_n \lambda A_n$ is computable. Therefore $(A_n)$ is a total Solovay test.
\end{proof}

Polynomial time randomness is incomparable to Schnorr randomness. On the one hand there are polynomial time random reals which are computable, and these cannot be Schnorr random. On the other hand, one can see from the proof of \cite[Theorem 7.11.6]{Downey.Hirschfeldt:book} that there is a real which is Schnorr random and not polynomial time random.


\newtheorem*{thm*}{Theorem}
\section{Rute: No-randomness-from-nothing for \\ computable randomness}

(By Jason Rute, Pennsylvania State University.) Given a randomness notion, e.g.~Schnorr randomness or Martin-Löf
randomness, there are two properties that one would naturally expect
to hold---randomness preservation and no-randomness-from-nothing.
We will show the second holds for computable randomness. 

Let $\mu$ be a computable probability measure on $\{0,1\}^{\mathbb{N}}$.
Assume $T\colon\{0,1\}^{\mathbb{N}}\rightarrow\{0,1\}^{\mathbb{N}}$
is a total computable map (later we will work with $\mu$-a.e.~computable
maps), and that $\mu_{T}$ is the push-forward measure given by $\mu_{T}(A)=\mu(T^{-1}(A))$. 

The first desirable property is \emph{randomness preservation} or
\emph{randomness conversation}:
\begin{quote}
If $X\in\{0,1\}^{\mathbb{N}}$ is $\mu$-random, then $T(X)$ is $\mu_{T}$-random.
\end{quote}
The intuition behind this property is that if we use one random object
$X$ to generate another object $T(X)$ (for example, use a sequence
of independent $1/2$-weighted coin-flips into simulate a sequence
of independent $1/4$-weighted coin-flips), then $T(X)$ is random
as well. This property is easy to prove for most of the standard randomness
notions---such as Kurtz randomness, Schnorr randomness, Martin-Löf
randomness, weak $n$-randomness, and $n$-randomness---as well as
a few of the more exotic notions---such as difference randomness,
Demuth randomness, and Oberwolfach randomness. However, it was shown
by Bienvenu and Porter \cite{Bienvenu.Porter:12}, as well as Rute \cite{Rute:14}
that it does not hold for computable randomness.

The second desirable property is \emph{no-randomness-from-nothing}:
\begin{quotation}
If $X\in\{0,1\}^{\mathbb{N}}$ is $\mu_{T}$-random, then $X=T(Y)$
for some $\mu$-random $Y\in\{0,1\}^{\mathbb{N}}$.
\end{quotation}
The idea behind this property is that all randoms in the push-forward
measure come from a random source. Or conversely, one cannot condense
a set of non-random objects into a random object. This property is
also known as \emph{surjectivity} since if randomness preservation
holds, then no-randomness-from-nothing is equivalent to saying that
$T$ is a surjective map when the domain and range of $T$ are restricted
to the randoms. It is known that no-randomness-from-nothing holds
for Martin-Löf randomness \cite{Bienvenu.Porter:12}, weak-$n$ randomness
($n\geq2$) and $n$-randomness {[}the last two are personal communication from Laurent
Bienvenu{]}. Bienvenu and Porter\cite{Bienvenu.Porter:12} asked if no-randomness-from-nothing
also holds for computable randomness and Schnorr randomness.

\begin{thm*}[No-randomness-from-nothing for computable randomness.]
If $\mu$ is a computable probability measure on $\{0,1\}^{\mathbb{N}}$,
$T\colon\{0,1\}^{\mathbb{N}}\rightarrow\{0,1\}^{\mathbb{N}}$ is a
$\mu$-a.e.~computable map, and $Y\in\{0,1\}^{\mathbb{N}}$ is $\mu_{T}$-computably
random, then $Y=T(X)$ for some $\mu$-computably random $X\in\{0,1\}^{\mathbb{N}}$.
\end{thm*}
This result can be generalized by replacing $\{0,1\}^{\mathbb{N}}$
with other computable metric spaces and $T$ with effectively measurable
maps---Schnorr layerwise computable maps. These generalizations will
appear in an upcoming paper. Also in that paper, I will give a very
natural class of maps such that randomness preservation holds for
computable randomness, and no-randomness-from-nothing holds for Schnorr
randomness.

First some notation and definitions. Let $\varepsilon$ be the empty
string. For a measure $\mu$ on $\{0,1\}^{\mathbb{N}}$ we will use
the notation $\mu(\sigma):=\mu([\sigma])$ and for a measure $\mu$
on $\{0,1\}^{\mathbb{N}}{\times}\{0,1\}^{\mathbb{N}}$, we will use
the notation $\mu(\sigma\times\tau):=\mu([\sigma]\times[\tau])$.
We say a measure $\mu$ on $\{0,1\}^{\mathbb{N}}$ is computable if
$\sigma\mapsto\mu(\sigma)$ is computable (and similarly for measures on $\{0,1\}^{\mathbb{N}}\times\{0,1\}^{\mathbb{N}}$). 

For computable randomness we will use an integral test characterization.
\begin{definition}[{Rute \cite{Rute:14}}]
Let $\mu$ be a computable measure on $\{0,1\}^{\mathbb{N}}$. A
sequence $X\in\{0,1\}^{\mathbb{N}}$ is \emph{$\mu$-computably random}
iff $t(X)<\infty$ for all lower semicomputable functions $t\colon\{0,1\}^{\mathbb{N}}\rightarrow[0,\infty]$
and all computable measures $\nu$ on $\{0,1\}^{\mathbb{N}}$ such
that 
\begin{equation}
\int_{[\sigma]}t\, d\mu\leq\nu(\sigma)\qquad(\sigma\in\{0,1\}^{*}).\label{eq:test-pair1}
\end{equation}

\end{definition}
A similar definition can be used to define computable randomness on
$\{0,1\}^{\mathbb{N}}{\times}\{0,1\}^{\mathbb{N}}$.
\begin{definition}
Let $\mu$ be a computable measure on $\{0,1\}^{\mathbb{N}}{\times}\{0,1\}^{\mathbb{N}}$.
A pair $(X,Y)\in\{0,1\}^{\mathbb{N}}{\times}\{0,1\}^{\mathbb{N}}$
is \emph{$\mu$-computably random} iff $t(X,Y)<\infty$ for all lower
semicomputable functions $t\colon\{0,1\}^{\mathbb{N}}{\times}\{0,1\}^{\mathbb{N}}\rightarrow[0,\infty]$
and all computable measures $\nu$ on $\{0,1\}^{\mathbb{N}}\times\{0,1\}^{\mathbb{N}}$
such that 
\begin{equation}
\int_{[\sigma]\times[\tau]}t\, d\mu\leq\nu(\sigma\times\tau)\qquad(\sigma,\tau\in\{0,1\}^{*}).\label{eq:test-pair2}
\end{equation}
\end{definition}

This next proposition justifies the previous definition.
\begin{prop}
Let $\mu$ be a computable measure on $\{0,1\}^{\mathbb{N}}{\times}\{0,1\}^{\mathbb{N}}$.
A pair $(X,Y)\in\{0,1\}^{\mathbb{N}}{\times}\{0,1\}^{\mathbb{N}}$
is $\mu$-computably random iff $X\oplus Y$ is computably random
on the pushforward measure of $\mu$ under the map $(A,B)\mapsto A\oplus B$.\end{prop}
\begin{proof}
Call this pushforward measure $\mu'$. Given any test pair $t,\nu$
satisfying (\ref{eq:test-pair2}) with $\mu$, consider the test pair
$t',\nu'$ where $t'(X\oplus Y)=t(X,Y)$ and $\nu'$ is the pushforward
of $\nu$ under the map $(A,B)\mapsto A\oplus B$. Then $t',\nu'$
satisfies (\ref{eq:test-pair1}) with $\mu'$. Conversely, any test
pair $t,\nu$ satisfying (\ref{eq:test-pair1}) can be translated
into a test pair $t,\nu$ satisfying (\ref{eq:test-pair2}) with $\mu$.\end{proof}
\begin{definition}
A partial map $T\colon\{0,1\}^{\mathbb{N}}\rightarrow\{0,1\}^{\mathbb{N}}$
is said to be $\mu$-a.e.~computable for a probability measure $\mu$
if $T$ is partial computable and $\mu(\operatorname{dom}T)=1$.\end{definition}
\begin{lemma}[{Rute \cite{Rute:14}}]
\label{lem:isomorphism}Let $\mu$ be a computable probability measure
on $\{0,1\}^{\mathbb{N}}$, let $F\colon\{0,1\}^{\mathbb{N}}\rightarrow\{0,1\}^{\mathbb{N}}$
be a $\mu$-a.e.~computable map, and let $G\colon\{0,1\}^{\mathbb{N}}\rightarrow\{0,1\}^{\mathbb{N}}$
be a $\mu_{F}$-a.e.~computable map such that 
\[
G(F(X))=X\ \ \mu\text{-a.e.}\qquad\text{and}\qquad F(G(Y))=Y\ \ \mu_{F}\text{-a.e.}
\]
Then $F$ and $G$ both preserve computable randomness and if $X$
is $\mu$-computably random and $Y$ is $\mu_{F}$-computably random
then 
\[
G(F(X))=X\qquad\text{and}\qquad F(G(Y))=Y.
\]

\end{lemma}

\begin{lemma}
\label{lem:graph}Let $\mu$ be a computable probability measure on
$\{0,1\}^{\mathbb{N}}$, let $T\colon\{0,1\}^{\mathbb{N}}\rightarrow\{0,1\}^{\mathbb{N}}$
be a $\mu$-a.e.~computable map, and let $(\mathrm{id},T)$ be the
map $X\mapsto(X,T(X))$. If $(X,Y)$ is $\mu_{(\mathrm{id},T)}$-computably
random, then $X$ is $\mu$-computably random and $Y=T(X)$.\end{lemma}
\begin{proof}
Notice that $(\mathrm{id},T)$ and its inverse $(X,Y)\mapsto X$
satisfy the conditions of Lemma~\ref{lem:isomorphism}. Therefore
if $(X,Y)$ is $\mu_{(\mathrm{id},T)}$-computably random, then $X$
is $\mu$-computably random. And by the composition $(X,Y)\mapsto X\mapsto(X,T(X))$,
we have that $Y=T(X)$.
\end{proof}
The main idea of this blog entry is as follows. The measure $\mu_{(\mathrm{id},T)}$
is supported on the graph of $T$, 
\[
\{(X,T(X)):X\in\{0,1\}^{\mathbb{N}}\}.
\]
Therefore, given a $\mu_{T}$-computably random $Y$, by this last
lemma, it is sufficient to find some $X$ which makes $(X,Y)$
$\mu_{(\mathrm{id},T)}$-computably random. That is the goal of the
rest of this article.
\begin{lemma}[{Folklore \cite[Theorem~7.1.3]{Downey.Hirschfeldt:book}}]
\label{lem:doob}Assume $X\in\{0,1\}^{\mathbb{N}}$ is $\mu$-computably
random and $\nu$ is a computable measure. Then the following limit
converges, 
\[
\lim_{n}\frac{\nu(X{\upharpoonright_{n}})}{\mu(X{\upharpoonright_{n}})}.
\]

\end{lemma}

\begin{lemma}
Let $\mu$ be a computable probability measure on $\{0,1\}^{\mathbb{N}}\times\{0,1\}^{\mathbb{N}}$
with second marginal $\mu_{2}$ (that is $\mu_{2}(\tau)=\mu(\varepsilon\times\tau)$.)
Assume $Y\in\{0,1\}^{\mathbb{N}}$ is $\mu_{2}$-computably random.
The following properties hold.
\begin{enumerate}
\item For a fixed $\tau\in\{0,1\}^{*}$, $\mu(\cdot\times\tau)$
is a probability measure,
\[
\mu(\sigma0\times\tau)+\mu(\sigma1\times\tau)=\mu(\sigma\times\tau).
\]

\item For a fixed $\sigma\in\{0,1\}^{*}$, $\mu(\sigma\times\cdot)$
is a probability measure.
\item The following limit converges 
\[
\mu(\sigma|Y):=\lim_{n}\frac{\mu(\sigma\times Y{\upharpoonright_{n}})}{\mu_{2}(Y{\upharpoonright_{n}})}.
\]

\item The function $\mu(\cdot|Y)$ defines a probability measure
\[
\mu(\sigma0|Y)+\mu(\sigma1|Y)=\mu(\sigma|Y).
\]

\item For a continuous function $f$, if $f^{Y}=f(\cdot,Y)$ then
\[
\int f^{Y}\, d\mu(\cdot|Y)=\lim_{n}\frac{\int_{[\varepsilon]\times[Y{\upharpoonright_{n}}]}f\, d\mu}{\mu_{2}(Y{\upharpoonright_{n}})}.
\]

\item For a nonnegative lower semicontinuous function $t$, if $t^{Y}=t(\cdot,Y)$
then
\[
\int t^{Y}\, d\mu(\cdot|Y)\leq\lim_{n}\frac{\int_{[\varepsilon]\times[Y{\upharpoonright_{n}}]}t\, d\mu}{\mu_{2}(Y{\upharpoonright_{n}})}.
\]

\end{enumerate}
\end{lemma}
\begin{proof}
(1) and (2) are apparent. Then (3) follows from (2) and Lemma~\ref{lem:doob}.
Then (4) follows from (1) and the definition of $\mu(\cdot|Y)$.

As for (5), first consider the case where $f$ is a step function
of the form $\sum_{i=0}^{k}a_{i}\mathbf{1}_{[\sigma_{i}]}$. This
case follows from (4). Since such step functions are dense in the
continuous functions under the norm $\|f\|=\sup_{X}f(X)$, we have
the result.

As for (6), $t=\sum_{k}f_{k}$ for a sequence of continuous nonnegative
$f^{k}$. Then we apply the monotone convergence theorem (MCT) for
integrals and Fatou's lemma for sums,
\begin{multline*}
\int t^{Y}\, d\mu(\cdot|Y)=\int\sum_{k}f_{k}^{Y}\, d\mu(\cdot|Y)\overset{\textrm{MCT}}{=}\sum_{k}\int f_{k}^{Y}\, d\mu(\cdot|Y)\\
=\sum_{k}\lim_{n}\frac{\int_{[\varepsilon]\times[Y{\upharpoonright_{n}}]}f_{k}\, d\mu}{\mu_{2}(Y{\upharpoonright_{n}})}\overset{\textrm{Fatou}}{\leq}\lim_{n}\sum_{k}\frac{\int_{[\varepsilon]\times[Y{\upharpoonright_{n}}]}f_{k}\, d\mu}{\mu_{2}(Y{\upharpoonright_{n}})}=\lim_{n}\frac{\int_{[\varepsilon]\times[Y{\upharpoonright_{n}}]}t\, d\mu}{\mu_{2}(Y{\upharpoonright_{n}})}.
\end{multline*}
\end{proof}
\begin{definition}
Let $\mu$ be a measure on $\{0,1\}^{\mathbb{N}}$ which may not be
computable. A sequence $X\in\{0,1\}^{\mathbb{N}}$ is \emph{blind $\mu$-Martin-Löf random}
if $t(X)<\infty$ for all lower semicomputable $t$ such that $\int t\, d\mu<\infty$.
\end{definition}
Finally, we have the tools to find some $X$ to pair with $Y$. What
follows is a variation of van Lambalgen's theorem.
\begin{lemma}
\label{lem:vL}Let $\mu$ be a computable measure on $\{0,1\}^{\mathbb{N}}{\times}\{0,1\}^{\mathbb{N}}$.
Let $Y$ be $\mu_{2}$-computably random. Let $X$ be blind $\mu(\cdot|Y)$-Martin-Löf
random. Then $(X,Y)$ is $\mu$-computably random.\end{lemma}
\begin{proof}
Assume $Y$ is $\mu_{2}$-computably random, but $(X,Y)$ is not $\mu$-computably
random. Then there is a lower semicomputable function $t$ and a computable
measure $\nu$ such that $t(X,Y)=\infty$ and $\int_{[\sigma]\times[\tau]}t\, d\mu\leq\nu(\sigma\times\tau)$.

Let $t^{Y}=t(\cdot,Y)$. Then $t^{Y}$ is lower semicomputable. Moreover,
since $t$ is lower semicontinuous, we have that
\[
\int t^{Y}\, d\mu(\cdot|Y)\leq\lim_{n}\frac{\int_{[\varepsilon]\times[Y{\upharpoonright_{n}}]}t\, d\mu}{\mu_{2}(Y{\upharpoonright_{n}})}\leq\lim_{n}\frac{\nu(\varepsilon\times Y{\upharpoonright_{n}})}{\mu_{1}(Y{\upharpoonright_{n}})}
\]
where the right-hand side converges to a finite value by Lemma~\ref{lem:doob}
since $Y$ is $\mu_{2}$-computably random. Therefore, $X$ is not
blind Martin-Löf random since $\int t^{Y}\, d\mu(\cdot|Y)<\infty$
and $t^{Y}(X)=t(X,Y)=\infty$.\end{proof}
\begin{thm}[No-randomness-from-nothing for computable randomness.]
If $\mu$ is a computable probability measure on $\{0,1\}^{\mathbb{N}}$,
$T\colon\{0,1\}^{\mathbb{N}}\rightarrow\{0,1\}^{\mathbb{N}}$ is a
$\mu$-a.e.~computable map, and $Y\in\{0,1\}^{\mathbb{N}}$ is $\mu_{T}$-computably
random, then $Y=T(X)$ for some $\mu$-computably random $X\in\{0,1\}^{\mathbb{N}}$.\end{thm}
\begin{proof}
Let $\nu=\mu_{(\mathrm{id},T)}$, and let $X$ be blind $\nu(\cdot|Y)$-Martin-Löf
random (there are $\nu(\cdot|Y)$-measure-one many, so there is at
least one). By Lemma~\ref{lem:vL}, $(X,Y)$ is $\mu_{(\mathrm{id},T)}$-computably
random. By Lemma~\ref{lem:graph}, $X$ is computably random and
$Y=T(X)$.\end{proof}

\part{Higher randomness}

\section{Monin: Separating higher weakly-2-random from $\Pi^1_1$-random}
Written by Benoit Monin in December 2013, joint work with Laurent Bienvenu and Noam Greenberg.

\begin{theorem} There is a higher weakly 2-random $X$ that is not $\Pi^1_1$-random. \end{theorem}
For background see  the 2013 Logic Blog \cite{LogicBlog:13}

We say that a higher $\Delta^0_2$ approximation $\{X_s\}_{s \in \wck}$ of a sequence $X$ is \textit{wicked} if for any computable ordinal $s$ we have that $X$ is not in the closure of $\{X_t\}_{t < s}$. A sequence $X$ is \textit{wicked} if it has a \textit{wicked} higher $\Delta^0_2$ approximation. It is easy to prove that for every wicked sequence $X$ we have $\omega_1^X>\wck$:

\begin{lemma} 
Let $\{X_s\}_{s \in \wck}$ be a wicked $\Delta^0_2$ approximation of $X$. Then $\omega_1^X>\wck$.
\end{lemma}
\begin{proof}
Let $f:\omega \rightarrow \wck$ be the $\Pi^1_1(X)$ function defined by setting $f(n)$ to the smallest ordinal $s$ so that $X_s \rest n = X \rest n$. As the function is total it is $\Delta^1_1(X)$. The fact that the approximation is wicked directly implies that $\sup_n f(n) = \wck$, and then $\omega_1^X>\wck$.
\end{proof}

 The goal of this section is to show that there is a wicked weakly-2-random sequence. Let $\{S_i\}_{i \in \omega}$ be an enumeration of all the higher $\Sigma^0_2$ sets. For each $S_i$ and each $j$ let us define the $\Sigma^1_1$ closed set $F_{i,j}$ so that $S_i=\bigcup_{j} F_{i, j}$.\\

\noindent \textbf{Sketch of the proof:}\\

We will build the sequence $X$ as a limit point of sequences $\{X_s\}_{s < \wck}$. Each $X_s$ is built as a limit of a sequence of strings $\{\s^n_s\}_{n < \omega}$ of bigger and bigger length.\\ 

At each stage we will ensure that $X_s$ is in some sense \textit{higher weakly-2-random at stage $s$}. What do we mean by this ? For some $s$ and some $n$, as long as $\lambda(S_n[s])=1$, we believe that $X_s$ should belong to $S_n[s]$. If at some point we have $\lambda(S_n[s])<1$ then $n$ is removed from the set of indices that we use to make $X_s$ weakly-2-random.\\ 

Concreatly we have at each stage $s$ a set of indices $\{e_n\}_{n \in \omega}$ which are initialized at stage 0 with $e_n=n$. Suppose that at stage $s$ we have for each $n$ that $\lambda(S_{e_n}[s])=1$. Then it is easy to build a $\Delta^1_1$ sequence $X_s$ in $\bigcap_n S_{e_n}[s]$:\\ 

We start by an arbitrary $\s^0$ and suppose that $\lambda(F_{e_{0},d_{0}} \cap \s^0) \geq \varepsilon_0/2$. Assuming that for some $n$ we have $\lambda(\bigcap_{k \leq n} F_{e_k,d_k} \cap \s^n) \geq \varepsilon_0/2$, we then continue recursively the construction as follow:\\

\begin{enumerate}
\item[Step 1:] We find one strict extention $\s^{n+1}_s$ of $\s^n_s$ so that $\lambda(\bigcap_{k\leq n} F_{e_k,d_k} \cap \s^{n+1}_s)[s] \geq \varepsilon_{n+1}$.\\

\item[Step 2:] We search for a closed set $F_{e_{n+1},d_{n+1}}$ inside $S_{e_{n+1}}$ so that $\lambda(\bigcap_{k\leq n+1} F_{e_k,d_k} \cap \s^{n+1}_s)[s] \geq \varepsilon_{n+1} / 2$. \\
\end{enumerate}

This way we build the whole sequence $X_s$, and the measure requirement that we satisfy at each step allow us to conclude that $X_s \in \bigcap_n S_{e_n}[s]$. To have that the $X_s$ converge to some $X$, we have to keep the chosen strings and closed sets at stage $s+1$ equal if possible to those of stage $s$. When do we have to change them ? three things can happend :\\


\begin{enumerate}

\item We might have $\lambda(S_{e_n})[s]=1$ for all $s < t$ but $\lambda(S_{e_n})[t]<1$.\\

\item We might have $\lambda(\bigcap_{k\leq n} F_{e_k,d_k} \cap \s^{n})[t] \geq \varepsilon_n / 2$ and $\lambda(\bigcap_{k\leq n} F_{e_k,d_k} \cap \s^{n+1})[s] \geq \varepsilon_{n+1}$ for all $s< t$ but $\lambda(\bigcap_{k\leq n} F_{e_k,d_k} \cap \s^{n+1})[t] < \varepsilon_{n+1}$\\

\item We might have $\lambda(\bigcap_{k\leq n} F_{e_k,d_k} \cap \s^{n+1})[t] \geq \varepsilon_{n+1}$ and $\lambda(\bigcap_{k\leq n+1} F_{e_k,d_k} \cap \s^{n+1})[s] \geq \varepsilon_{n+1} / 2$ for all $s<t$ but $\lambda(\bigcap_{k\leq n+1} F_{e_k,d_k} \cap \s^{n+1})[t] < \varepsilon_{n+1} / 2$\\
\end{enumerate}

If (1) happends then the index $e_n$ is setted to some index $a$ so that $\lambda(S_a)=1$ and then $e_n$ can change at most once. If (2) happends it is the responssability of the string $\s^{n+1}$ to change. If (3) happends it is the responsability of the index $d_{n+1}$ to change.\\ 

In (2), it is easy to chose $\varepsilon_{n+1}$ in a way that a string which fits the requirement always exists. Then it is enougth to pick them in lexicographic order to be sure that we have one which fits the requirement after a finite number of change.\\

In (3) as long as $\lambda(S_{e^{n+1}})=1$ we are also sure that we will change only finitely often of index $d_{n+1}$. However if $\lambda(S_{e^{n+1}})<1$ it can happend that $d_{n+1}$ will change infinitely often at stage $s_1<s_2<\dots$ before we have at stage $t=\sup_n s_n$ that $\lambda(S_{e^{n+1}})[t]<1$ (and then that $e^{n+1}$ is set to $a$).\\

There is nothing we can do to prevent those infinitely many changes, which could lead as well to infinitely many changes of the string $\s_{n+1}$. However we can still ensure that if this happends, the string $\s_{n}$ is banished, so that the final $X$ is still wicked.\\

To do so, we need to chose the $\varepsilon_n$ in a way that at least two extentions satisfying the requirements exists. So that if one is banished, still one remains. before starting the formal proof, we give a lemma which will help us to chose the right $\varepsilon_n$ in the construction.
%
%
%
%
%
%

\begin{lemma} \label{nicelemma}
let $\s$ be a string and $F$ a closed set so that $\lambda(F \cap [\s]) \geq 2^{-|\s|-n}$. Then there is at least $2$ extensions $\tau_1,\tau_2$ of $\s$ of length $|\s|+n+1$ so that for $i \in \{1,2\}$ we have $\lambda(F \cap [\tau_i]) \geq 2^{-|\tau_i| - n-1}$.
\end{lemma}
\begin{proof}
We have that $\lambda(F \cap [\s])=\sum_{\tau_i}\lambda(F \cap [\tau_i])$. Suppose that for strictly less than $2$ extensions of length $|\s|+n+1$ we have $\lambda(F \cap [\tau_i]) \geq 2^{-|\tau_i| - n-1}$. Then we have:
$$
\begin{array}{rcl}
\sum_{\tau_i}\lambda(F \cap [\tau_i])&\leq&2^{-|\s|-n-1} + (2^{n+1}-1)2^{-|\tau_i| - n-1}\\
&\leq&2^{-|\s|-n-1}+2^{n+1}2^{-|\s| - 2n - 2} - 2^{-|\s| - 2n-2}\\
&\leq&2^{-|\s|-n-1}+2^{-|\s| - n - 1} - 2^{-|\s| - 2(n+1)}\\
&<&2^{-|\s|-n}
\end{array}
$$
which is a contradiction.
\end{proof}

\noindent We now start the formal proof.\\

\noindent \textbf{The construction}:\\

Let $\{S_i\}_{i \in \omega}$ be an enumeration of all the higher $\Sigma^0_2$ sets. For each $S_i$ and each $j$ let us define the $\Sigma^1_1$ closed set $F_{i,j}$ so that $S_i=\bigcup_{j} F_{i, j}$.\\

Set $q_{0}=1$ and $m_0=1$ then recursively set $m_{n+1}=q_n+2$ and $q_{n+1}=2q_{n}+4-m_n$. Finally, let $a$ be an integer so that $\lambda(S_a)=1$. \\

At stage $0$, at substage $0$, set $P^0_0$ to be the set of strings of length $m_0$, ordered lexicographically. Then set $\s^{0}_{0}$ to be the first string of $P^0_0$, set $e^0_0=0$, set $d^0_0=0$. At substage $n+1$, set $P^{n+1}_0$ to be the set of strings of length $m_{n+1}$, ordered lexicographically. Then set $\s^{n+1}_{0}$ to be the first string of $P^{n+1}_0$ extending $\s^{n}_{0}$, set $e^{n+1}_0=n+1$, set $d^{n+1}_0=0$.\\

At successor stage $s+1$, set by convention $\s^{-1}_{s+1}=\epsilon$ and $F_{e^{-1}_{s+1},d^{-1}_{s+1}} = 2^\omega$ (in order to avoid making a difference between substage $0$ and other substages). At substage $n$ (starting from $0$), if $\lambda(S_{e^{n}_{s}})[s+1]=1$ set $e^n_{s+1}=e^n_{s}$ and $P^n_{s+1}=P^n_{s}$, otherwise set $e^n_{s+1}$ to be equal to $a$ and $P^n_{s+1}=P^n_{s}-\{\s^n_{s}\}$. Then set $\s^n_{s+1}$ to be the first string $\s$ of $P^n_{s+1}$ extending $\s^{n-1}_{s+1}$ so that $\lambda(\bigcap_{k \leq n-1} F_{e^{k}_{s+1},d^{k}_{s+1}} \cap [\s])[s+1] \geq 2^{-q_{n}}$.\\

Still at substage $n$, if $\lambda(\bigcap_{k \leq n} F_{e^{k}_{s},d^{k}_{s}} \cap [\s^{n}_{s+1}])[s+1] \geq 2^{-q_n-1}$ set $d^n_{s+1}=d^n_{s}$. Otherwise set $d^n_{s+1}$ to be the smallest $d$ bigger than $d^n_{s}$ so that $\lambda(\bigcap_{k \leq n-1} F_{e^{k}_{s+1},d^{k}_{s+1}} \cap F_{e^{n}_{s+1},d} \cap [\s^{n}_{s+1}])[s+1] \geq 2^{-q_n-1}$.\\


Finally after every substage, define $A_{s+1}$ to be the unique element in $\bigcap_n [\s^n_{s+1}]$.\\

At limit stage $s$, compute a sequence of ordinals so that $s=\sup_m s_m$. By convention, like at successor stage set $\s^{-1}_{s}=\epsilon$ and $F_{e^{-1}_{s+1},d^{-1}_{s+1}} = 2^\omega$. Then for each $n\geq 0$ set $e^n_{s}$ to be the convergence value of $\{e^n_{s_m}\}_{m \in \omega}$ and set $P^n_s$ to be the convergence value of $\{P^n_{s_m}\}_{m \in \omega}$. (among other things we will have to prove that we always have convergence). \\

At substage $n$, if $\{\s^{n}_{s_m}\}_{m \in \omega}$ does not converge set $\s^{n}_s$ to be the first string of $P^n_s$ extending $\s^{n-1}_s$, otherwise set $\s^{n}_s$ to be the convergence value. If $\{d^{n}_{s_m}\}_{m < \omega}$ does not converge, set $d^{n}_{s}$ to $0$, otherwise set it to its convergence value.\\

Finally after every substage, define $A_{s}$ to be the unique element in $\bigcap_n [\s^n_{s}]$.\\
\\
\textbf{The verification}:\\

\noindent Claim 1: For every $n$ the sequence $\{e^n_t\}_{t < \wck}$ can change at most once:\\

Fix some $n$ and suppose that there is a stage $s$ for which $e^n_{s+1} \neq e^n_{s}$. Let $s$ be the smallest such stage. By construction we have that $e^n_{s+1}=a$ and as $\lambda(S_{a})=1$, we can prove by induction on stages that for any $t>s$ we have $e^n_{t}=a$. this implies that for the sequence $\{e^n_t\}_{t < \wck}$ can change at most once.\\

\noindent Claim 2: At every stage $s$ the set $\{e^n_s\ |\ n \in \omega\}$ contains all the indices $e$ such that $\lambda(S_e)=1$.\\

It is true at stage $0$ as $\{e^n_0\ |\ n \in \omega\}=\omega$. Suppose it is true at stage $s$. At stage $s+1$, if $e^n_{s+1}$ is different from $e^n_{s}$ it is by construction because $\lambda(S_{e^n_{s}})[s+1]<1$. Then using induction hypothesis, Claim 2 is also true at stage $s+1$. At limit stage $s=\sup_n s_m$, we have using claim 1 that the sequence $\{e^n_{s_m}\ |\ n \in \omega\}_{s_m < s}$ converges to $\{e^n_{s}\ |\ n \in \omega\}_{s}$ and then by induction hypothesis we have that Claim 2 is true at limit stage $s$.\\

\noindent Claim 3: At every stage $s$ the sequence set $\bigcap_n [\s^n_s]$ contains one and only one element $A_s$.\\ 

This is obvious by construction, as at any stage $s$ and for any $n$ the string $\s^{n+1}_s$ extends $\s^n_s$.\\

\noindent Claim 4: For every stage $s$, any string $\tau$ of size $m_n$ and any closed set $F$ such that $\lambda(F \cap [\tau]) \geq 2^{-q_{n}-1}$, there is a string $\s\in P^{n+1}_{s+1}$ which extends $\tau$ so that $\lambda(F \cap [\s]) \geq  2^{-q_{n+1}}$.\\

Suppose that $\lambda(F \cap [\tau]) \geq 2^{-q_{n}-1}$ for $|\tau|=m_n$. Then $\lambda(F \cap [\tau]) \geq 2^{-m_n-(q_{n}+1-m_n)}$. Using lemma \ref{nicelemma} we have two strings $\tau_1$ and $\tau_2$ of length $m_n+(q_{n}+1-m_n)+1=q_n+2=m_{n+1}$ so that for $i \in \{1,2\}$ we have $\lambda(F \cap [\tau_i]) \geq 2^{-(q_{n}+2)-(q_{n}+1-m_n)-1} = 2^{-2q_{n}-4+m_n} = 2^{-q_{n+1}}$. Also as $\tau_1$ and $\tau_2$ are of length $m_{n+1}$ they belong to $P^{n+1}_{0}$. By construction and by Claim 1, at any stage $s$ we have that $P^{n+1}_{0}$ contains at most one more string than $P^{n+1}_{s+1}$. Then at any stage $s$ we have at least one string $\s \in P^{n+1}_{s+1}$ which extends $\tau$ and so that $\lambda(F \cap [\s]) \geq  2^{-q_{n+1}}$.\\

\noindent Claim 5: The sequence $A_s$ is \textit{wicked}:\\

For a given $n$, let $B_n$ be the set of stages $s$ so that $\s^{n}_s=\s^{n}$ if the sequence $\{\s^{n}_s\}_{s < \wck}$ converges to $\s^{n}$ and $B_n=\emptyset$ if the sequence $\{\s^{n}_s\}_{s < \wck}$ does not converge. Let us prove by induction on $n$ that:\\
\begin{enumerate}
\item[(1)] The set $B_n$ is not empty.\\
\item[(2)] The sequence of sets $\{d^{k}_t\ |\ k\leq n\}_{t \in B_{n}}$ can change at most finitely often.\\
\item[(3)] The sequence of sets $\{\s^{n+1}_t\}_{t \in B_n}$ can change at most finitely often.\\
\end{enumerate}

The sentence (1) is true for $n=-1$, as $\s^{-1}_t=\epsilon$ for every $t$. The sentence (2) is true for $n=-1$ as the the sequence $\{d^{-1}_t\}_{t < \wck}$ never changes. The sentence (3) is true for $n=-1$ as by claim 1, the sequence $\{\s^0_t\}_{t < \wck}$ can change at most once.\\ 

Let us suppose that (1), (2) and (3) are true up to $n$ and let us show that (1) is true at level $n+1$. By induction hypothesis (1) at level $n$ we have a stage $s$ so that for all $t>s$ we have $\s^n_t=\s^n_s$. Also by induction hypothesis (3) at level $n$ we have a stage $r>s$ so that for all $t>r$ we have $\s^{n+1}_r=\s^{n+1}_t$. Then $(1)$ is true at stage $n+1$.\\

Let us suppose that (1), (2) and (3) are true at level $n$, that (1) is true at level $n+1$ and let us show that (2) is true at level $n+1$. By induction hypothesis (2) at level $n$ and the fact that $B_{n+1} \subseteq B_{n}$ we have that :\\
\begin{enumerate}
\item[(4)] The sequence of sets $\{d^{k}_t\ |\ k\leq n\}_{t \in B_{n+1}}$ can change at most finitely often.\\
\end{enumerate}
By definition of $B_{n+1}$ we have that:\\
\begin{enumerate}
\item[(5)] The sequence $\{\s^{n+1}_t\}_{t \in B_{n+1}}$ can change at most finitely often.\\
\end{enumerate}
Also by claim 1 we have that:\\
\begin{enumerate}
\item[(6)] The sequence of sets $\{e^{k}_t\ |\ k \leq n+1\}_{t \in B_{n+1}}$ can change at most finitely often.\\
\end{enumerate}

Suppose for contradiction that the sequence $\{d^{n+1}_t\}_{t \in B_{n+1}}$ changes infinitely often. From (4), (5) and (6) we can deduce that there are two stages $r_1 < r_2$ so that :\\
\begin{enumerate}
\item[(7)] The sequence of sets $\{d^{k}_t\ |\ k \leq n\}_{t \in [r_1, r_2[ \cap B_{n+1}}$ never changes.
\item[(8)] The sequence of sets $\{e^{k}_t\ |\ k \leq n+1\}_{t \in [r_1, r_2[ \cap B_{n+1}}$ never changes.
\item[(9)] The sequence $\{\s^{n+1}_t\}_{t \in [r_1, r_2[ \cap B_{n+1}}$ never changes.
\item[(10)] The sequence $\{d^{n+1}_t\}_{t \in [r_1, r_2[ \cap B_{n+1}}$ changes infinitely often.
\item[(11)] We have $r_1 \in B_{n+1}$ and we have that $r_2$ is the smallest stage bigger than $r_1$ and infinitely many stages in $B_{n+1}$ such that (10) is true.\\
\end{enumerate}

Let us call the unique values for (7), (8) and (9) respectively $\{d^{k}\ |\ k \leq n\}$, $\{e^{k}\ |\ k \leq n+1\}$ and $\s^{n+1}$ (note that $\s^{n+1}$ is the convergence value for the sequence $\{\s^{n+1}_t\}_{t < \wck}$). From (7), (8) and (9) we deduce that:\\
\begin{enumerate}
\item[(12)] $\lambda(\bigcap_{k \leq n} F_{e^{k}, d^{k}} \cap [\s^{n+1}])[t+1] \geq 2^{-q_{n+1}}$ for $t \in [r_1,r_2[\cap B_{n+1}$.\\
\end{enumerate}

As otherwise, by construction, at least one of the three sequences (7), (8) or (9) would not be constant. Also if $\lambda(S_{e^{n+1}})[r_2]=1$, we deduce from (12) that there is necessarily some $d$ so that for any $t \in [r_1,r_2[\cap B_{n+1}$ we have $\lambda(\bigcap_{k \leq n} F_{e^{k}, d^{k}} \cap F_{e^{n+1}, d} \cap [\s^{n+1}])[t+1] \geq q_{n+1}/2$. But using (10) we know that no such $d$ exists, as otherwise, by construction, the sequence of (10) would not change infinitely often. Then we have that $\lambda(S_{e^{n+1}})[r_2]<1$. Also $r_2$ is the smallest such stage, as otherwise, by construction the sequence (8) would not be constant. Also from (9), by construction and the fact that $r_2$ is a limit stage (by definition of $r_2$) we have that $\s^{n+1}=\s^{n+1}_{r_2}$ and then that $r_2 \in B_{n+1}$. Then still by construction we have that $P^{n+1}_{r_2+1}$ does not contain $\s^{n+1}_{r_2}$ at stage $r_2+1$, wich contradicts the fact that $r_2 \in B_{n+1}$. Then (2) is true at level $n+1$.\\

Suppose now that (1), (2) and (3) are true at level $n$, suppose also that (1) and (2) are true at level $n+1$ and let us show that (3) is true at level $n+1$. Suppose for contradiction that the sequence $\{\s^{n+2}_t\}_{t \in B_{n+1}}$ changes infinitely often. From induction hypothesis (2) at level $n+1$ we have that: \\

\begin{enumerate}
\item[(13)] The sequence of sets $\{d^{k}_t\ |\ k \leq n+1\}_{t \in B_{n+1}}$ can change at most finitely often.\\
\end{enumerate}

As before we have that $(5)$ is true and that $(6)$ is true, this time for $k \leq n+2$. We can deduce from that and from (13) that there are two stages $r_1 < r_2$ so that :\\

\begin{enumerate}
\item[(14)] The sequence of sets $\{d^{k}_t\ |\ k \leq n+1\}_{t \in [r_1, r_2[ \cap B_{n+1}}$ never changes.\\
\item[(15)] The sequence of sets $\{e^{k}_t\ |\ k \leq n+2\}_{t \in [r_1, r_2[ \cap B_{n+1}}$ never changes.\\
\item[(16)] The sequence $\{\s^{n+1}_t\}_{t \in [r_1, r_2[ \cap B_{n+1}}$ never changes.\\
\item[(17)] The sequence $\{\s^{n+2}_t\}_{t \in [r_1, r_2[ \cap B_{n+1}}$ changes infinitely often.\\
\end{enumerate}

Let us call the unique values for (14), (15) and (16) respectively $\{d^{k}\ |\ k \leq n+1\}$, $\{e^{k}\ |\ k \leq n+2\}$ and $\s^{n+1}$. From (14), (15) and (16) we deduce that:\\

\begin{enumerate}
\item[(18)] $\lambda(\bigcap_{k \leq n+1} F_{e^{k}, d^{k}} \cap [\s^{n+1}])[t+1] \geq 2^{-q_{n+1}-1}$ for $t \in [r_1,r_2[\cap B_{n+1}$.\\
\end{enumerate}

As otherwise, by construction, at least one of the three sequences (14), (15) or (16) would not be constant. Now, using Claim 4 with $\bigcap_{k \leq n+1} F_{e^{k}, d^{k}}[r_2]$ as the closed set $F$, we deduce that there is at least one string $\s \in P^{n+2}_{r_2}$ so that for any $t \in [r_1,r_2[\cap B_{n+1}$ we have $\lambda(\bigcap_{k \leq n+1} F_{e^{k}, d^{k}} \cap [\s])[t+1] \geq 2^{-q_{n+2}}$ which directly contradicts (17).\\

We now use (1) and (3) in order to prove Claim 5. By (1) we have that $X_s$ converges to $X$. By (3) we have for each $n$ that the sequence $\{\s^{n+1}_t\ |\ \s^{n+1}_t \succ \s^{n}\}_{t < \wck}$ can change at most finitely often. Suppose now for contradiction that there is some $t$ so that $X$ is a limit point of $\{X_s\}_{s < t}$. \\


Then using Claim 6 (see further) we have that $X$ is weakly-2-random and then not $\Delta^1_1$, so $X_t \neq X$. This implies that for some $n$ we have infinitely many stages $s_0<s_1<s_2< \dots$ so that $\s^n_{2i}=X \rest {m_n}$ and $\s^n_{2i+1} \neq X \rest {m_n}$. Let $n$ be the smallest such integer. We have that $n \geq 1$ as $\{\s_0\}_{s < \wck}$ changes at most finitely often. Also by minimality of $n$ we have some $k$ so that $\s^{n-1}_{s_j}=X \rest {m_{n-1}}$ for all $j\geq k$, which implies that the sequence $\{\s^{n}_t\ |\ \s^{n}_t \succ \s^{n-1}\}_{t < \wck}$ changes infinitely often. Thus $X$ is wicked.\\

\noindent Claim 6: The sequence $X$ is weakly-2-random:\\

At convergence we have a sequence of closed set $\{F^n=F_{e^n,d^n}\}_{n \in \omega}$ so that each $F_{e^i,d^i}$ is a component of $S_{e^i}$. Consider also a sequence of ordinal $\{s_n\}_{n \in \omega}$, cofinal in $\wck$, and the closed set $A=\{X_{s_n}\}_{n < \omega} \cup \{X\}$. We have for any $n$ and any stage $t$ that the intersection of closed set $\bigcap_{s < t} \bigcap_{m\leq n} F_m[s] \cap A$ is not empty because by construction there is some stage $s_k$ so that $X_{s_k} \in \bigcap_{m\leq n} F_m[t]$. Then also $\bigcap_{t < \wck} \bigcap_{n} F_n[t] \cap A$ is not empty, and as $X$ is the only non $\Delta^1_1$ point of $A$ we have that $X$ belongs to $\bigcap_{n} F_n$. This combined with Claim 2 allow us to conclude that $X$ is weakly-2-random.

\section{Yu: A proof of a result by Kripke on the upward closure of null sets of hyper degrees}

In \cite{Sacks:69}, the following result is  credited to Kripke.
\begin{theorem}\label{theorem: Kripke}
Let  $A$ be  a null set of hyperdegrees such that $\mathbf{0}\not\in A$. Then the set $\mathcal{U}_h(A)=\{x\mid \exists y\in A(x\geq_h y)\}$ is also null.
\end{theorem}

Theorem \ref{theorem: Kripke} significantly strengthened Sacks' result that the set of hyperdegrees hyperarithmetically above a fixed nonhyperarithmetic real is null.   No published proof of the theorem was  known. Here we give a short proof based on a recent result due to Chong and Yu, a hyperarithmetic version of Demuth's theorem on downward closure of random degrees.

\begin{proof}
In \cite{Chong.Yu:nd}, it was proved that if $x$ is $\Pi^1_1$ random and $y\leq_h x$, then $y\equiv_h y_0$ for some $\Pi^1_1$ random real $y_0$. By partially relativizing the proof, we have that if $x$ is $\Pi^1_1(z)$ random and $y\leq_h x$, then $y\equiv_h y_0$ for some $\Pi^1_1(z)$ random real $y_0$.

Now suppose that $A$ is a null set of hyperdegrees such that $\mathbf{0}\not\in A$ and  $\mathcal{U}_h(A)=\{x\mid \exists y\in A(x\geq_h y)\}$ is not null (may be even non-measurable!). Then there is some real $z$ so that $A$ is a subset of non-$\Pi^1_1(z)$-random reals. Since $\mathcal{U}_h(A)$ is not null, it contains some real $x$ which is $\Pi^1_1(z)$-random. So there must be some $y\in A$ so that $y\leq_h x$. By Chong and Yu's result, there must be some $y_0\equiv_h y$ so that $y_0$ is $\Pi^1_1(z)$-random. But then $y_0\in A$, a contradiction. 
\end{proof}

The higher Demuth's Theorem can also be used to construct a null maximal antichain of hyperdegrees and Turing degrees. This answers  question 5.3 of Jockusch in \cite{Yu06}. Also it can be used to solve question 5.1 of Jockusch in \cite{Yu06}.

\section{Yu: higher randomness over continuous measures}

The following result is due to Chong and Yu.
\begin{theorem}\label{theorem: cy theorem on pi11 measure}
 $x\in \{z\mid z\in L_{\omega_1^z}\}$ if and only if $x$ cannot be $\Pi^1_1$-random relative to any continuous measure.
\end{theorem}

By Theorem \ref{theorem: cy theorem on pi11 measure}, if $x\not\in L_{\omega_1^x}$, then $x$ is $\Pi^1_1$-random relative to some measure $\lambda$. However, the proof of Theorem \ref{theorem: cy theorem on pi11 measure} does not provide an explicit  construction of $\lambda$ provided $x\not\in L_{\omega_1^x}$.

If $x$ is hyperarithmetically equivalent to a $\Pi^1_1$-random real, then we can find a constructive proof. By (and the proof of ) higher Demuth Theorem and Reimann-Slaman Theorem, we have the following result.
\begin{proposition}\label{proposition: cy theorem pi11 random relative to hyp measure}
For any real $x$, $x\equiv_h r$ for some  $\Pi^1_1$-random real $r$ if and only if there is a hyperarithmetic continuous measure $\lambda$ so that $x$ is $\Pi^1_1$-random relative to measure $\lambda$.  
\end{proposition}

The proof of Proposition \ref{proposition: cy theorem pi11 random relative to hyp measure} is constructive. In other words, if $r$ is $\Pi^1_1$ random and $x\equiv_h r$, then $\lambda$ can be constructed explicitly. So the proposition shed some light on the question whether there is a constructive proof of Theorem \ref{theorem: cy theorem on pi11 measure}. Also it can be asked that whether there is a ``level by level" proof of Theorem \ref{theorem: cy theorem on pi11 measure}. For example, if a continuous measure  $\lambda$ has a $\Pi^1_1$-presentation, then which reals can be $\Pi^1_1$-random relative to $\lambda$?

In summary, we have the following question:
\begin{question}
Is there a constructive proof of Theorem \ref{theorem: cy theorem on pi11 measure}? 
\end{question}

The question can also be lifted to constructible level. We have the following result.
\begin{proposition}\label{proposition: cy theorem l random relative to hyp measure}
For any real $x$, $x\equiv_L r$ for some  random real $r$ over $L$ if and only if there is a  continuous measure $\lambda$ in $L$ so that $x$ is random over $L$ relative to measure $\lambda$.
\end{proposition}

Then the further question is: Suppose that $ZF+AD+DC$, which reals can be $L[\hat{\lambda}]$-random relative to a continuous measure $\lambda$? Where $\hat{\lambda}$ is a presentation of $\lambda$.

Note that the set $$A=\{x\mid x \mbox{ cannot be }L[\lambda]\mbox{-random for any continuous measure }\lambda\}$$ is $\Pi^1_3$.  So $A$ is a $\Pi^1_3$ countable set under the assumption.

\begin{proposition}
$0^{\sharp}\in A$.
\end{proposition}
\begin{proof}
Assume that $PD$. By Simpson, there is a function $j:2^{\omega}\to Ord$ (for example, $j(x)=(\kappa^+)^{L[x]}$) where $\kappa=\aleph_1$ so that
\begin{enumerate}
\item $x\leq_L y\rightarrow j(x)\leq j(y)$; and
\item $x\leq_L y\rightarrow (j(x)<j(y)\leftrightarrow x^{\sharp}\leq_L y)$.
\end{enumerate}

So for any continuous measure $\lambda\not\geq_L 0^{\sharp}$, $\lambda\oplus 0^{\sharp}\geq_L \lambda^{\sharp}$. So $0^{\sharp}$ cannot be random over $\lambda$.

\end{proof}

So $A$ is neither $\Pi^1_2$ nor $\Sigma^1_2$.

By Proposition \ref{proposition: cy theorem l random relative to hyp measure}, $A$ is closed under $L$-equivalence relation. However, $A$ is not closed under $\Delta_3^1$-equivalence relation. For example, there is a $\Delta^1_3$ $L$-random real Turing below $0^{\sharp}$.

\begin{question}
Assuming $ZF+AD+DC$, give a characterization of $A$.
\end{question}

\section{Kjos-Hanssen and Yu: Mixed lowness for higher randomness}
          All the results below were proved by Kjos-Hanssen and Yu in 2011 when Yu was visiting Kjos-Hanssen in San Diego.
          
	The reals that are low for the implications from Martin-L\"of randomness to Kurtz randomness and Schnorr randomness have in the past been characterized as the non-DNR degrees and the c.e.\ traceable degrees. Here we obtain the analogous results in the setting of hyperarithmetic reducibility: lowness for the implications from $\Pi^1_1$-ML-randomness to $\Delta^1_1$-randomness and $\Delta^1_1$-Kurtz randomness equal $\Pi^1_1$-traceable hyperdegrees and non-$\Pi^1_1$-DNR hyperdegrees, respectively.

	\begin{definition}
	$X$ is $\Delta^1_1$-SNR if there exists a function $f\le_h X$ such that for each total $\Delta^1_1$ function $\Phi$, there do not exist infinitely many $e$ with $f(e)=\Phi(e)$.
	\end{definition}

	\begin{definition}
	$X$ is $\Pi^1_1$-DNR if there exists a function $f\le_h X$ such that there do not exist infinitely many $e$ with $f(e)=\Phi_e(e)$.
	Here $\Phi_e(i)=R_e(i)$ is a universal $\Pi^1_1$ function. Let $R_e$, $e\in\omega$ be an effective list of all the partial recursive predicates. Then $R_e(i)=m\iff \forall X\exists n$ $R_e(X,n,m,i)$.
	\end{definition}

	\begin{theorem}
	$\Delta^1_1$-SNR implies $\Pi^1_1$-DNR.
	\end{theorem}
	\begin{proof}
	Suppose $X$ is not $\Pi^1_1$-DNR. So there are infinitely many $e$ with $f(e)=\Phi_e(e)$.

	Case 1: $X\ge_h O$. Impossible.

	Case 2: $X\not\ge_h O$. Then $\omega_1^X=\omega_1^{\text{CK}}$. Then find $e_1,e_2,\ldots$ with $f(e_n)=\Phi_{e_n}(e_n)\downarrow$ by stage $\alpha$. Bound $\alpha<\omega_1^{\text{CK}}$ ($\Sigma^1_1$-bounding). Then $\exists^\infty e$ $f(e)=\Phi(e):=\Phi_e(e)[\alpha]$. So $X$ is not $\Delta^1_1$-SNR.
	\end{proof}

	\begin{figure}
	\begin{tabular}{|c|c|c|c|}
	\hline
	 Low($y$,$x$)			& $\Delta^1_1$-Kurtz 	&  $\Delta^1_1$-rand 		& $\Pi^1_1$-MLR				\\
	\hline
	$\Delta^1_1$-Kurtz 				&  $\Pi^{1}_{1}$-dominated, not $\Pi^{1}_{1}$-DNR \cite{Kjos.Nies.Stephan.Yu:10}	&  			& 				\\
	\hline
	$\Delta^1_1$-rand 				& not $\Pi^{1}_{1}$-DNR 	& $\Pi^{1}_{1}$-traceable 			& 	\\ \hline
	$\Pi^1_1$-MLR		&  \emph{not $\Pi^{1}_{1}$-DNR} 	& \emph{$\Pi^{1}_{1}$-traceable} 			& $\Delta^1_1$				\\
	\hline
	\end{tabular}
	\caption{A complete picture of lowness for higher randomness between $\Pi^{1}_{1}$-ML-randomness and $\Delta^{1}_{1}$-Kurtz randomness. \emph{Results from the present paper.}}
	\end{figure}

	\begin{theorem}
	Each $\Pi^1_1$-DNR computes an infinite subset of a $\Pi^1_1$-MLR set.
	\end{theorem}
	\begin{proof}
	Follows the general plan of Barmpalias, Miller, Nies \cite{Brattka.Miller.ea:nd} (Miller's result). 
	\end{proof}

	\begin{theorem}
	Low($\Pi^1_1$-MLR, $\Delta^1_1$-random) equals $\Pi^1_1$-traceable.
	\end{theorem}
	\begin{proof}
	This follows the method of Kjos-Hanssen, Nies, and Stephan 2005 \cite{Kjos.ea:05}.
	\end{proof}

	We also show $2-\Pi^1_1$-DNR implies $\Delta^1_1$-high implies $\ge_h$ $\Pi^1_1$-ML-random implies $\ge_h$ $\Pi^1_1$-DNR, and that the middle implication does not reverse.

	We first prove that $2-\Pi^1_1$-DNR implies $\Delta^1_1$-high.
	\begin{lemma}\label{lemma: pa computes pi01}
	If $x$ is $2-\Pi^1_1$-DNR and $A\subseteq 2^{\omega}$ is a nonempty $\Sigma^1_1$-closed set, then there must be some $z\leq_h x$ in $A$.
	\end{lemma}
	We may simulate the proof that every 2-DNR degree computes a member of every nonempty $\Pi^0_1$-set to show Lemma \ref{lemma: pa computes pi01}.
	
	\begin{lemma}\label{lemma: a pi01 only contains high}
	There is a nonempty $\Sigma^1_1$-closed set $A$ in which every member Turing computes all hyperarithmetic reals.
	\end{lemma}
	\begin{proof}
	Note that there is a recursive function $f:\omega\to \omega$ so that for any $n\in \mathcal{O}$, $\Phi_{f(n)}$ is a $tt$-reduction so that $\Phi^{\mathcal{O}}_{f(n)}=H_n$.
	
	Let $\sigma\in T$ if and only if $\sigma\in 2^{<\omega}$ and for any $n$, if $n\in \mathcal{O}$, then $\Phi^{\sigma}_{f(n)}\prec H_n$. So $T$ is a $\Sigma^1_1$-subtree of $2^{<\omega}$.
	
	Note that $[T]$ is not empty since $\mathcal{O}\in [T]$.
	
	If $x\in [T]$, then for any $n\in \mathcal{O}$, $\Phi_{f(n)}^x$ is total since $\Phi_{f(n)}$ is $tt$-reduction. Moreover,$\Phi_{f(n)}^x =H_n$.
	
	Let $A=[T]$.
	\end{proof}
	Now if $x$ is $2-\Pi^1_1$, then by Lemma \ref{lemma: pa computes pi01} and \ref{lemma: a pi01 only contains high}, there must be some real $z\leq_h x$ so that every hyperarithmetic real recursive in $z$. Then $z$ must be $\Delta^1_1$-high.
	
	To see that $\Delta^1_1$-high implies $\geq_h$$\Pi^1_1$-ML-random, just note that every schnorr random relative to a $\Delta^1_1$-high real is $\Delta^1_1$-random. Then it follows by the dichotomy of randomness. No $\Pi^1_1$--random can be $\Delta^1_1$-high. Otherwise, $r^{(\alpha)}$ Turing computes all hyperarithemtic reals for some $\Pi^1_1$-random real $r$ and recursive ordinal $\alpha$. However, by the lowness of $\Pi^1_1$-randomness,  $r^{(\alpha)}\oplus \emptyset^{(\beta)}<_T r^{(\alpha+1)}\leq_T r\oplus \emptyset^{(\beta)}$ for some recursive ordinal $\beta$, a contradiction. So the reverse direction fails.
	
	Clearly every $\Pi^1_1$-$ML$-random hyperarithmetically computes a $\Pi^1_1$-DNR function. \footnote{Kihara proved that the reverse direction doesn't hold?}

\newpage
\part{Metric spaces}
\section{Gavruskin and Nies: Left-c.e.\ metric spaces}

Alexander  Gavruskin and Andr\'e Nies worked during a Tokerau Beach retreat Nov 2013. They finished their research a few months later.  The paper has now appeared in the Lobchevskii Math. Journal.

Consider a Polish metric space $\+ M = (M,d, \seq {p_k} \sN k)$ where $d$ is the distance function, and $\seq {p_k}$ is a designated dense sequence.  The $p_i$ are called the special points. To say that $\+ M$ is   computable   means that $d(p_i,p_k)$ is a computable real uniformly in $i,k$.
 There has  been recent  work on computable metric spaces (for instance by Melnikov), and on $\PI 1$ equivalence relations  on $\NN$ by Ianovski et al. \cite{Ianovski.Miller.ea:14}. The following generalizes both concepts.

\begin{definition} We say that  $\+ M = (M,d, \seq {p_k} \sN k)$ is  a left-c.e.\ [right-c.e.] metric space  if $d(p_i,p_k)$ is a left-c.e.\  real  [right-c.e.\ real] uniformly in numbers $i,k \in \NN$. We write $d_s(p,q)$ for the distance of special points $p, q$ at stage $s$.
\end{definition}
We mainly study left-c.e.\ spaces. For them, intuitively speaking, the distance between points can increase over time. 
\begin{example}  {\rm Let  $\beta$ be  a left-c.e.\ real. Then $[0, \beta]$ is a  left-c.e.\ metric space; so is the circle of radius $\beta$ with the Euclidean metric inherited from $\RR^2$. In fact $[0, \beta]$  has a computable presentation even if $\beta$ is noncomputable. (Note that  this presentation is not  effectively compact.) In contrast,   the circle of radius $\beta$ does not have a computable presentation at all if $\beta$ is noncomputable.  To see this, given $n$, find points $q_0, \ldots, q_{n-1}$ among the $p_i$   such that   $\fa i,k \, |d(q_i,q_{i+1})- d(q_k, q_{k+1})| \le \tp{-n}$ (where addition is taken mod $n$). This means the points $q_i$ are the corners of an $n$-gon up to small error. Hence   $\sum_{i<n} d(q_i,q_{i+1})$ converges to $2\pi  \beta$ effectively. So $2\pi \beta$ is computable.  (In more detail,  each side of the $n$-gon has length within $\tp{-n}$ of $2\beta \cos (\pi/n)$, so the error for the $n$-th approximation is at most $n \tp{-n}  + 2 \beta   |\pi  - n \cos (\pi/n)|$.)

Let $E$ be a $\PPI$ equivalence relation on $\NN$. Define $d(x,y)= 1- 1_E(\la x, y \ra)$, namely $0$ if $Exy$ and $1$ otherwise. Clearly this metric makes $\NN$ a left-c.e.\ metric space (with $p_k = k$). This space   has a computable presentation, though not uniformly. Note that the $\PPI$ equivalence relations can be uniformly identified with  left-c.e.\ metric space where the set of possible distances is $\{0,1\}$.}
\end{example}

\begin{definition}  A \emph{Cauchy name}  in a   metric space with distinguished dense sequence  $\+ M = (M,d, \seq {p_k} \sN k)$ is a function $g$ on $\NN$ such that $d(p_{g(i)}, p_{g(k)}) \le \tp {-i} $ for each $i \le k$. \end{definition}

 The Cauchy names in a left-c.e.\ metric space  form a $\PPI$ subclass of Baire space. This together with the examples seems to suggest it  is more natural to look at left-c.e.\ spaces,  rather than at right-c.e.\ ones. 

\begin{definition}  We say that an  isometric embedding $g$  from a   left-c.e.\ [right-c.e.] metric spaces  $\+ M = (M,d, \seq {p_k} \sN k)$ to another $\+ N = (N,d, \seq {q_k} \sN k)$  is \emph{computable} if uniformly in $k$ one can compute a Cauchy name for  $g(p_k)$.  \end{definition} 

  It is not hard to see that there is a universal   object in the class of  right-c.e.\ -metric spaces with respect to computable isometric embeddings.  We now obtain the corresponding result for left-c.e.\ spaces with a fixed bound on the diameter.

\begin{thm}  \label{thm: Universal left-ce} Let $\gamma>0$ be a left-c.e.\ real. Within the class of  left-c.e.\ metric spaces\ of diameter at most  $\gamma$,  there is a left-c.e.\ metric space\ $\+ U$  which is universal with respect to  computable isometric embeddings. \end{thm}

\begin{proof}  
For the duration of the proof, a  \emph{left-c.e.\ pre-metric function} will be a symmetric function $h\colon \NN \times \NN \to [0,\gamma]$ such that  $h(v,v) = 0 $ and $h(v,w)$ is a left-c.e.\ real, uniformly in $v,w$.  Such a function is presented by its c.e.\ undergraph $G=\{ \la v,w, q \ra \colon q \in \QQ^+_0 \lland v\le w \lland q \le h(v,w)\}$.  Let $h_t(x,y) = \max \{ q \colon \,  \la v,w, q \ra \in G_t\}$ so that  $h(v,w) = \sup_t h_t(v,w)$. 
\begin{lemma} Given a pre-metric  function $h$, one can effectively determine a c.e.\ set $W$ that is an initial segment of $\NN$, and a pre-metric function $g$ such that $g$ satisfies the triangle inequality on $W$; if $h$ satisfies the triangle inequality then $W = \NN$ and $g = h$. \end{lemma}

\begin{proof}[Proof of lemma]  We   extend  the direct construction of a $\PI 1$-complete equivalence relation   in \cite[Subsection 13.1]{LogicBlog:12}. (Note that the proof in \cite{Ianovski.Miller.ea:14} differs from the original proof.)

Define a partial computable sequence of stages by $t_{-1} = -1$, $t_0 = 0$ and   
	
\begin{eqnarray}  t_{i+1}  &  \simeq   \mu t > t_i &  \fa v,w \le t_i  \big [  h_t (v,w) - \wt h_t(v,w) \le \tp{-i}\big ]   \lland  \label{ln-1} \\  
   & &   \fa v,w \le t_{i-1}  [ \wt h_t(v,w) \ge  \wt h_{t_i} (v,w)],  \label{ln-2}
\end{eqnarray}
	where  
	\[ \wt h_t(v,w) = \inf  \{ \sum_{r<t_i}  h_t(q_r, q_{r+1}) \colon \, q_0, \ldots, q_{t_i} \le t_i \lland q_0= v \lland q_{t_i} = w\}. \]
Note that $\wt h_t(v,w) \le h_t(v,w)$, and $\wt h_t $ satisfies the triangle inequality on $[0,t_i]$.  
Informally, the sequence of stages can only be continued at a stage $t > t_i$ if the ``improved version''  $\wt h_t$ that satisfies the triangle inequality is close enough to $h_t$; furthermore, for values $v,w \le t_{i-1}$ its value  at $t$ must be at least the one at the last relevant stage $t_i$.  

If we define $t_{i+1}$ we also put $(t_{i-1}, t_i]$ into  W. Finally, for  $v,w \in \NN$, if $v\in W \lland w\in W$  we let  
\begin{equation*}    g(v,w)  = \sup \{ \wt h_{t_{i+1}}(v,w)  \colon \, v,w \le t_i\}, 	\end{equation*}
	and otherwise $g(v,w) = 0$.

Each $\wt h_{t_i+1}$ satisfies the triangle inequality on $[0,t_{i}]$, so from the monotonicity in $t_i$ it is clear that the  pre-metric function $g$   satisfies the triangle inequality on $W$. 
\begin{claim} Suppose that $h$  satisfies the triangle inequality. Then there are infinitely many stages $t_i$, so $W=\NN$ and $g=h$. \end{claim} 
We proceed by induction on $i$. Suppose that  $t_i$ is  defined.  Assume that  $t_{i+1}$ remains undefined.
Given $\epsilon >0$  let $s > t_i$ be a stage such that $\fa  x,y \le t_i \, [ h(x,y) - h_s(x,y) \le  \epsilon /t_i$. Then for each $v,w \le t_i$ we have  $ h(v,w) \le \wt h_s(v,w)  +  \epsilon$. Thus   $\lim_{s> t_i} \wt h_s(v,w) = h(v,w)$. This shows that both \eqref{ln-1} and \eqref{ln-2} are satisfied for sufficiently large $s$. Hence $t_{i+1}$ will be defined, contradiction. This concludes the inductive step and the proof of the lemma.   	
 \end{proof}	
Clearly one   can effectively list all the  left-c.e.\ pre-metric functions as $h_0, h_1, \ldots$  Let $\seq{W_n, g_n }\sN n  $ be the pairs of  a c.e.\ initial segment of $\NN$   and a  pre-metric function and given by the lemma. Note that $(W_n, g_n)$ presents a metric space for each $n$.  The required universal left-c.e.\ metric space $\+ U$ is an   effective  disjoint union of these spaces, where the distance between points in different components is $\gamma$. More formally,  let $f$ be a 1-1 computable function with range $\bigcup_n \{n\} \times W_n$.  Let $p_k =k$ be the special points of $\+ U$, and define  

\vsp

$d_\+ U(i,k)   = \gamma$ if $f(i)_0 \neq f(k)_0$; otherwise, let 

$d_\+ U(i,k) = g_p(f(i)_1, f(k)_1)$, where $p= f(i)_0 = f(k)_0$.

\vsp

\n Given a left-c.e.\ metric space $\+ M = (M,d, \seq {q_r} \sN r)$, there is  $n$ such that  $h_n(r,s) = d(q_r,q_s)$. This $h_n$ satisfies the triangle inequality. So  $q_r \to f^{-1}(\la n,r \ra)$ is a computable isometric embedding $\+ M  \to \+U$ as required. Note that the range of the embedding only contains special points of $\+ U$, which can be thought of as Cauchy names that are constant.
\end{proof} 
From the proof of the foregoing theorem for $\gamma =1 $, it is clear that $\{\la i,k \ra \colon \, d_\+ U (i,k) =0\}$ is a universal $\PPI$ equivalence relation under computable reducibility in the sense of \cite{Ianovski.Miller.ea:14}. In this way we have re-obtained their existence result for such an object.

The bound $\gamma$ on the diameter in Theorem~\ref{thm: Universal left-ce} is necessary. If we discard that bound, there is no universal object  in the class of all left-c.e. metric spaces with respect to computable isometric embeddings by the following fact.
\begin{proposition}
Let $\+ M = (M, d, \seq{p_k})$ be a left-c.e.\ metric space. There exists a left-c.e.\ ultrametric space $\+ A$ such that $\+ A$ can not be  isometrically embedded into $\+ M$. 
\end{proposition}

\proof We make $\+ A$ a discrete metric space with domain consisting solely of special points $q_0,q_1, \ldots$ We use the pair $q_{2n}, q_{2n+1}$ to destroy the $n$-th potential computable embedding into $\+ M$. Such an embedding  $f$ would be given by computably associating to $q_k$ a Cauchy name for $f(q_k)$. In particular, the map sending $q_k$ to the first component  of this Cauchy name is computable. Note that this component is a special point  $p_r$  of $\+ M$  which has distance at most $1$ from $f(q_k)$.

Let $\seq {\phi_n} \sN n$ be an effective listing of partial computable functions with domain an initial segment of $\NN$. By the discussion above, it suffices to meet the requirements

\bc $R_n: \phi_n \, \text{total} \,  \RA  d_\+ A (q_{2n}, q_{2n+1}) >   d(  \phi_n(q_{2n}), \phi_n(q_{2n+1}) )+ 2$. \ec

\constr At stage $s$, for  $n= 0, \ldots,  s$, if $\phi_n(q_{2n}), \phi_n(q_{2n+1})$ are defined, let 
\bc $v_{n,s}= d_s(  \phi_n(q_{2n}), \phi_n(q_{2n+1}) )+3$;  \ec 
otherwise  $v_{n,s}=0$. 
Set $d_{\+ A,s} (q_{2n}, q_{2n+1}) = v_{n,s}$ towards satisfying $R_n$. 

For each  $n\neq k$, and $a = 2n, 2n+1$, $b= 2k, 2k+1$, set 
\bc $d_{\+ A,s} (q_a, q_b ) = \max (v_{n,s}, v_{k,s})$. \ec

It is clear that the distances in $\+ A$ remain finite. Also  the ultrametric inequality holds because trivially for reals $u,v,w$   we have  \bc $\max (u,w)  \le \max ( \max (u,v),  \max (v,w))$. \ec We apply this to distances of the form  $d_{\+ A} (q_{2n}, q_{2n+1} )$ to get the ultrametric inequality between points from three different pairs. 
Finally,   each requirement $R_n$ is met. 
\endproof
%
%
%
%

The structure of our universal objects appears to be  quite arbitrary. We ask whether an extra condition can be satisfied that would make them unique under computable isomorphism. This condition is effective homogeneity in the sense of Fra\"\i ss\'e theory.

\begin{question} Can a universal $\PPI$ equivalence relation, or a left-c.e.\ metric space, be effectively homogeneous? \end{question}

The bound $\gamma$ on the diameter in Theorem~\ref{thm: Universal left-ce} is necessary. If we discard that bound, there is no universal object  in the class of all left-c.e. metric spaces with respect to computable isometric embeddings by the following fact.
\begin{proposition}
Let $\+ M = \la M, d, \seq{p_k}$ be a left-c.e.\ metric space. There   exist a left-c.e.\ ultrametric space $\+ A$ such that $\+ A$ can not be  isometrically embedded into $\+ M$. 
\end{proposition}

\proof We make $\+ A$ a discrete metric space with domain consisting solely of special points $q_0,q_1, \ldots$. We use the pair $q_{2n}, q_{2n+1}$ to destroy the $n$-th potential computable embedding into $\+ M$. Such an embedding  $f$ would be given by computably associate to $q_k$ a Cauchy name for $f(q_k)$. In particular, the map sending $q_k$ to the first component  of this Cauchy name is computable. Note that this component is a special point  $p_r$  of $\+ M$  which has distance at most $1$ from $f(q_k)$.

Let $\seq {\phi_n} \sN n$ be an effective listing of partial computable functions with domain an initial segment of $\NN$. By the discussion above, it suffices to meet the requirements

\bc $R_n: \phi_n \, \text{total} \,  \RA  d_\+ A (q_{2n}, q_{2n+1}) >   d(  \phi_n(q_{2n}), \phi_n(q_{2n+1}) )+ 2$. \ec

\constr At stage $s$, for  $n= 0, \ldots,  s$, if $\phi_n(q_{2n}), \phi_n(q_{2n+1})$ are defined, let 
\bc $v_{n,s}= d_s(  \phi_n(q_{2n}), \phi_n(q_{2n+1}) )+3$;  \ec 
otherwise  $v_{n,s}=0$. 
Set $d_{\+ A,s} (q_{2n}, q_{2n+1} = v_{n,s}$. 

For each  $n\neq k$, and $a = 2n, 2n+1$, $b= 2k, 2k+1$, set 
\bc $d_{\+ A,s} (q_a, q_b ) = \max (v_{n,s}, v_{k,s})$. \ec

It is clear that the distances in $\+ A$ remain finite. Also  the triangle inequality holds because trivially for reals $u,v,w$   we have $\max (u,w)  \le \max ( \max (u,v),  \max (v,w))$. We apply this to distances $d_{\+ A} (q_{2n}, q_{2n+1} )$.
Finally,   each requirement $R_n$ is met. 
\endproof
%
%
%
%

\section{Gavruskin, Ng and Nies: Limit points of c.e.\ sets of special points in a CMS} 
Gavruskin,  Ng and Nies worked in Auckland. They generalised a result of LeRoux and Ziegler.  
\begin{thm}
Let $(X,d)$ be a perfect computable metric space with special points $\{q_i\}_{i\in\omega}$ and $\mathcal{C}\subseteq X$ be $\Pi^0_1(\emptyset')$ with respect to $(X,d)$. Then there is a c.e. set of special points $D\subseteq\{q_i\}_{i\in\omega}$ such that $\mathcal{C}$ equals  the set of limit points of $D$ in $(X,d)$.
\end{thm}

\begin{proof}
We assume that $\{q_i\}_{i\in\omega}$ is a computable enumeration of the set of special points and fix a computable enumeration of all basic open balls $\{B_k\}_{k\in\omega}$, for instance, take $B_k=B(q_i,2^{-j})$ where $k=\langle i,j\rangle$. 

We also fix a $\Sigma^0_2$ enumeration of all basic open balls in $X-\mathcal{C}$, that is, we fix a $\Sigma^0_2$-set $I\subseteq\omega$ such that $\displaystyle X-\mathcal{C}=\cup_{k\in I} B_k$. We now define the desired set $D$ by doing the following for each~$i\in\omega$:

Wait for some stage $t>i$ such that there exists some special point $q$ satisfying:
\begin{itemize}
\item $q\in B_i$, and
\item For each $j\leq i$ such that $B_j=B(q',2^{-r})$ has not left $X-\mathcal{C}$ between stages $i$ and $t$, we require that $d(q,q')>2^{-r}-2^{-i}$.
\end{itemize}
If $q$ is found enumerate it in $D$, until we have enumerated $2^i$ many special points in $D$. We say that $B_i$ contributes $q$ to $D$. 

It is easy to check that these conditions are $\Sigma^0_1$, and so we produce a c.e. set $D$ of special points. Note that since the space is perfect, each $B_i$ either contributes nothing to $D$ or it contributes exactly $2^i$ many special points to $D$.

We now verify that the set of limit points of $D$ is $\mathcal{C}$. 

First suppose that  $x\in\mathcal{C}$. Since the space is perfect fix a sequence $\left\{q_{h(i)}\right\}_{i\in\omega}$ of distinct special points such that for every $i$, $d(q_{h(i)},x)<2^{-i}$. It suffices to show that for each $i$, the ball $B(q_{h(i)},2^{-i+1})$ contributes at least one point to $D$.

We proceed by contradiction. Suppose the ball $B_{l}=B(q_{h(i)},2^{-i+1})$ with $l=\langle h(i),i+1\rangle$ contributes nothing to $D$. Since $q_{h(m)}\in B_l$ for all $m>i$, it follows that for each $m>i$, we must have the failure of the second condition for some $j\leq l$ such that $B_j$ is permanently in the complement $X-\mathcal{C}$. Since there are only finitely many such $j\leq l$, there is some $j$ such that $B_j=B(q',2^{-r})$ is permanently in $X-\mathcal{C}$ where there are infinitely many $m>i$ such that $d(q_{h(m)},q')\leq 2^{-r}-2^{-l}$. However for each $m$ we have 
\begin{align*} d(x,q') &\leq d(x,q_{h(m)})+d(q_{h(m)},q')\\
&\leq 2^{-m}+(2^{-r}-2^{-l})\\
&= 2^{-r} -(2^{-l}-2^{-m}).
\end{align*}
By choosing a large enough  $m$ we conclude that $x\in B_j\subseteq X-\mathcal{C}$, a contradiction.

Now suppose that $x\not\in\mathcal{C}$. Fix a $j$ such that $x\in B_j=B(q,\delta)$. Suppose that $B_j$ is enumerated permanently in the complement $X-\mathcal{C}$ at stage $s>j$. Fix $\varepsilon >0$ such that $d(x,q)<\delta-\varepsilon$, and assume that $2^{-s}<\varepsilon$. We claim that after stage $s$, no special point $p\in B(q,\delta-\varepsilon)$ can be enumerated in $D$ for the sake of any $B_k$, $k>s$. Since each $B_0,B_1,\cdots B_s$ contributes finitely many points to $D$, $x$ cannot be a limit point of $D$.

Suppose on the contrary that at some stage larger than $s>j$ we have some special point $p\in B(q,\delta-\varepsilon)$ enumerated in $D$ for the sake of $B_k$, for some $k>s$. By construction we must have $d(q,p)>\delta-2^{-k}>\delta-2^{-s}>\delta-\varepsilon$, a contradiction.
\end{proof}


  
\part{Complexity of equivalence relations}
\section{Turetsky: Complete Equivalence Relations}

  We work in the framework of $m$-reductions between arithmetical  equalence relations. For background see e.g.\  \cite{Ianovski.Miller.ea:14}  and the references there.

\begin{prop}
For every oracle $X$, there is an equivalence relation on $\omega$ which is d.c.e.\ in $X$ and effectively complete for such relations.
\end{prop}

\begin{proof}
Given a binary relation $E$ which is d.c.e.\ in $X$, it suffices to uniformly construct an equivalence relation $F$ which is d.c.e.\ in $X$ and such that $F = E$ if $E$ is an equivalence relation.  We do this by copying $E$, but slowly (similar to the construction of a complete $\Pi^0_1$-equivalence relation    in \cite[Subsection 13.1]{LogicBlog:12}).  We keep a value $n_s$, which indicates how much of $E$ we are copying.

Begin by defining $F_0 = \{ (a,a) : a \in \omega\}$ and $n_0 = 0$.

At stage $s > 0$, if $E_s\! \upharpoonright_{n_{s-1}+1}$ is an equivalence relation, define
\[
F_s = E_s\! \upharpoonright_{n_{s-1}+1} \cup \{(a,a) : a \in \omega\}
\]
and $n_s = n_{s-1}+1$.  Otherwise, define $F_s = F_{s-1}$ and $n_s = n_{s-1}$.

Observe that $F$ is d.c.e.\ as described, since the only time we change $F$'s value on a pair is when we are copying $E$, and $E$ is d.c.e..  Further, it is easy to check by induction that every $F_s$ is an equivalence relation, and so $F$ is an equivalence relation.  Finally, if $E$ is an equivalence relation, $n_s$ will grow without bound, and so we will have $F = E$.
\end{proof}

Note that this construction works relative to any oracle, while the construction of a complete $\Pi^0_1$ equivalence relation does not.  Indeed, since there is no complete $\Pi^0_2$ equivalence relation, that other construction cannot be relativized to $\emptyset'$.  The difference between the two constructions is the reduction from $E$ to $F$: when $E$ is an equivalence relation, the reduction from $E$ to $F$ in this construction is the identity, and so is computable without $X$; when $E$ is an equivalence relation, the reduction from $E$ to $F$ in the $\Pi^0_1$ construction is a fast growing function that depends on the presentation of~$E$.

Note also that this construction can be generalized to $n$-c.e.\ sets, while the construction for $\Pi^0_1$ equivalence relations can be generalized to co-$n$-c.e.\ sets.  Similarly, the proof that there is no complete $\Pi^0_2$ equivalence relation can be generalized to show that there is no complete co-$n$-c.e.\ equivalence relation in $\emptyset'$.  The key difference is that with $n$-c.e.\ sets, elements begin outside the set, while for co-$n$-c.e.\ sets, they begin in.

\newpage
\part{Shen and Bienvenu: $\KPt$-trivial, $\KPt$-low and $\MLR$-low sequences- a tutorial}

 These notes came out of lectures  by L.~Bienvenu at Poncelet laboratory, CNRS, Moscow, notes by A.~Shen,  LIRMM, Montpellier, CNRS, UM2, on leave from IITP RAS, Moscow.

A remarkable achievement in algorithmic randomness and algorithmic information theory was the discovery of the notions of $\KPt$-trivial, $\KPt$-low and Martin-L\"of-random-low sets: three different definitions turns out to be equivalent for very non-trivial reasons. This paper, based on the course taught by one of the authors (L.B.) in Poncelet laboratory (CNRS, Moscow) in 2014, provides an exposition of the proof of this equivalence and some related results.

We assume that the reader is familiar with basic notions of algorithmic information theory (see, e.g., \cite{Shen:00} for introduction and \cite{Shen.etal:12} for more detailed exposition). More information about the subject and its history can be found in~\cite{Nies:book,Downey.Hirschfeldt:book}.

\section{$\KPt$-trivial sets: definition and existence}

Consider an infinite bit sequence and complexities of its prefixes. If they are small, the sequence is computable or almost computable; if they are big, the sequence looks random.   This idea goes back to 1960s and appears in the algorithmic information theory in different forms (Schnorr--Levin criterion of randomness in terms of complexities of prefixes, the notion of algorithmic Hausdorff dimension). The notion of $\KPt$-triviality is on the low end of this spectrum: we consider sequences that have prefixes of minimal possible (prefix) complexity:

\begin{definition}
A bit sequence $a_0a_1a_2\ldots$, is called \emph{$\KPt$-trivial} if it has minimal possible prefix complexity of its prefixes, i.e., if 
      $$
\KP(a_0 a_1\ldots a_{n-1})=\KP(n)+O(1).      
      $$
\end{definition}

\begin{itemize}
\item Note that $n$ can be reconstructed from $a_0\ldots a_{n-1}$, so $\KP(a_0\ldots a_{n-1})$ cannot be smaller than $\KP(n)-O(1)$. 
\item Every computable sequence is $\KPt$-trivial (since $a_0\ldots a_{n-1}$ can be computed given~$n$).
\item Similar definition for plain complexity has no sense, since this would imply that sequence $A$ is computable (it is enough to have $\KS(a_0 a_1\ldots a_{n-1})\le\log n +O(1)$ for computability, see, e.g.,\cite[problems 48 and 49]{Shen.etal:12}).
\end{itemize}

With prefix complexity we have a weaker property:  
\begin{theorem}
Every $\KPt$-trivial sequence is $\mathbf{0}'$-computable.
\end{theorem}

Here $\mathbf{0}'$ is the oracle for the halting problem.

\begin{proof}
Assume that $\KP((a)_n)=\KP(n)+O(1)$. Recall that $(a)_n=a_0a_1\ldots a_{n-1}$ is equivalent to $(n,(a)_n)$, and use the formula for the complexity of a pair:
   $$
\KP((a)_n)=\KP(n,(a)_n)=\KP(n)+\KP((a)_n\cnd n,\KP(n))
   $$
(with $O(1)$-precision); this means that $$\KP((a)_n\cnd n,\KP(n))=O(1).$$ So $(a)_n$ belongs to a $\mathbf{0}'$-computable (given $n$) list of $n$-bit strings that has size $O(1)$. Therefore, $a$ is a path in a $\mathbf{0}'$-computable tree of bounded width and is therefore $\mathbf{0}'$-computable. (Indeed, assume that the tree has $k$ infinite paths that diverge before some level $N$. At levels after $N$ we can identify all the paths, since all other vertices have finite subtrees above them, and we may wait until only $k$ candidates remain.)
\end{proof}

The existence of (non-computable) $\KPt$-trivial sets is not obvious, but not very difficult, even if we additionally require the set to be enumerable.

\begin{theorem}
There exists an enumerable undecidable $\KPt$-trivial set.
\end{theorem}

The existence of noncomputable $\KP$-trivial sequences was shown by Solovay in 1970s (answering Chaitin's question).

\begin{proof}
 We want $\KP((a)_n)$ to be (almost) equal to $\KP(n)$. In terms of a priori discrete probability, we want to match $\m(n)$ at the vertex $(a)_n$. In game terms: when opponent (who ``lower semicomputes'' $\m$) increases the weight of $n$ (the sum of his total weights never exceeds $1$), we should increase the weight of some vertex (=string) of length $n$, achieving the same (up to $O(1)$-factor) weight. Moreover, all these vertices should lie on a tree path that corresponds to a enumerable set, and the sum of our weights should be bounded by some constant.

It would be trivial for computable $a$: just fix some computable $a$ and then place at $(a)_n$ the same weight as the opponent uses for $n$. But we want $a$ to be non-computable. It is convenient to ensure that $a$ is a characteristic function of a simple set in Post's sense. 

Recall that a simple set $A$ is a set with infinite complement that has non-empty intersection with every $W_n$ that is infinite. Here by $W_n$ we denote $n$-th enumerable set in some natural numbering of all c.e.~sets. As in the classical Post's construction, we want for every $n$ to put in $A$ some element of $W_n$ greater than $2n$, and then forget about $W_n$. (The bound $2n$ is needed to guarantee that the set has infinite complement.) Post did it without reservations (as soon as some element greater than $2n$ is discovered in $W_n$, it is added to $A$), but now, when such an element is found, we have to pay something for it: when we add some number $u$ to $A$, the corresponding path $a$, the characteristic sequence for $A$, changes; the weight used on $(a)_u$ is lost, and should be recreated at the new place. Moreover, all the weights put at the extensions of $(a)_u$ are lost, too (since now these vertices are not on the path, and the new weights should be placed along the new path).  This lost amount can be called the \emph{cost} of the action.

Now we can explain the construction. Initially our set $A$ is empty, and the corresponding path $a$ is all zeros. While observing the growth of the discrete a priori probability $\m(n)$, we replicate the corresponding values along the path $a$. We also enumerate all $W_n$ in parallel, and add some element $u\in W_n$ to $A$, if three conditions are satisfied:
\begin{itemize}
\item $u>2n$;
\item $W_n$ was not satisfied before; 
\item the cost of this action is small (say, less than $2^{-n}$, so the total cost for all $n$ is bounded). 
\end{itemize}
Here the cost of the action is the total weight placed at the vertices in $a$ starting from level $u$ (this weight is lost and should be replicated along the new path).

In this way the total weight used by us is bounded: the lost weight is bounded by $\sum 2^{-n}$, and the replicated weight is bounded by $\sum \m(n)$.

If $W_n$ is infinite, it contains arbitrarily large elements, and the cost of adding $u$ is bounded by $$\m(u)+\m(u+1)+\m(u+2)+\ldots,$$ which is guaranteed to go below threshold for large $u$. So every infinite $W_n$ will be served at some moment.
\end{proof}

This proof can be represented in a game form. In such a simple case this looks like an overkill, but the same technique is useful in more complicated cases, so it is instructive to look at this version of the proof. The game field consists of the set of the natural numbers (\emph{lengths}), the full binary tree, and sets $W_1,W_2,\ldots$ (of natural numbers). The opponent increases the weights assigned to lengths: each length has some weight that is initially zero and can be increased by the opponent at any moment by any rational non-negative number; the only restriction is that the total weight (of all lengths) should not exceed~$1$. Also the opponent may add new elements to any of the sets $W_i$ (initially they are empty). We construct a path $a$ in the binary tree that is a characteristic sequence of some set $A$ (initially empty) by adding elements to $A$; we also increase the weights of vertices of the binary tree (in the same way as the opponent does for lengths); our total weight should not exceed~$2$.

One should also specify when the players can make moves; it is not important (the rules of the game always allow each player to postpone moves), but let us agree that the players make their moves in turns and every move is finite (finitely many weights of lengths and vertices are increased by some rational numbers, and finitely many new elements are added to $W_i$ and $A$). This is the game with full information, the moves of one player are visible to the other one.

The game is infinite, and the winner is determined in the limit (assuming that both player obey the weight restrictions). Namely, we win if
\begin{itemize}
\item for the limit path $a$ our weight of $(a)_n$ is not less than the opponent's weight of $n$;
\item for each $n$, if $W_n$ is infinite, then $W_n$ has a common element with $A$.
\end{itemize}
The winning strategy is as described: we match the opponent's weight along the current path, and also we add some $u$ to $A$ (changing the path and matching the opponent's weights along the new path) if $u$ belongs to $W_n$, is greater than $2n$ and the loss (our total weight along the current path above $u$) does not exceed $2^{-n}$. 

This is a computable winning strategy. Indeed, the limit weights of all lengths form a converging series, so if $W_n$ is infinite, it has some element that is greater than $2n$ and for which the loss (bounded by the tail of this series) is less than $2^{-n}$.

Imagine now that we use this computable winning strategy against the ``blind'' computable opponent that ignores our moves and just enumerates from below the a priori probability (as lengths' weights) and the sets $W_i$ (the list contains all enumerable sets). Then the game is computable, our limit $A$ is an enumerable simple set, and our weights for the prefixes of $a$ (and therefore $\m((a)_n)$, since the limit weights form a lower semicomputable semimeasure) match $\m(n)$ up to $O(1)$-factor.

\section{$\KPt$-trivial and $\KPt$-low sequences}

Now we know that non-computable $\KPt$-trivial sequences do exist.  Our next goal is to show that they are ``almost computable''. Namely, they are $\KPt$-low in the sense of the following definition.

Consider a bit sequence $a$; one can relativize the definition of prefix complexity using $a$ as a oracle (the decompressor may use the values of $a_i$ in its computation). For every oracle this relativized complexity $\KP^a$ does not exceed (up to $O(1)$ additive term) the non-relativized prefix complexity, since the decompressor may ignore the oracle. But it can be smaller or not, depending on the oracle.

\begin{definition}
A sequence $a$ is $\KPt$-low if $\KP^a(x)=\KP(x)+O(1)$.
\end{definition}

In other words, $\KPt$-low oracles are useless for compression (better to say, decompression) purposes. 

Obviously, computable oracles are low; the question is whether there exist non-computable low oracles. One can note that \emph{``classical'' undecidable sets, like the halting problem, are not low}: with oracle $\mathbf{0}'$ the table of complexities of all $n$-bit strings has complexity $O(\log n)$, but its non-relativized complexity is $n-O(1)$.  (One can also consider the relativized and non-relativized complexities of the prefixes of Chaitin's $\Omega$-numbers.) Note also that \emph{$\KPt$-low oracles are $\KPt$-trivial} (since $\KP^a ((a)_n)=\KP^a(n)+O(1)$: the sequence $a$ is computable in the presence of oracle $a$). 

It turns out that the reverse implication is true, and this is quite surprising. For example, one may note that \emph{the notion of a $\KPt$-low sequence is Turing-invariant} (i.e., depends only on the computational power of the sequence) but for $\KPt$-triviality there are no reasons to expect this (since the definition deals with prefixes). 

On the other hand, it is easy to see that \emph{if $a$ and $b$ are two $\KPt$-trivial sequences, then their join} (the sequence $a_0b_0a_1b_1a_2b_2\ldots$) \emph{is also $\KPt$-trivial}. Indeed, as we have mentioned, the $\KPt$-triviality of a sequence $a$ means that $\KP((a)_n\cnd n, \KP(n))=O(1)$; if also $\KP((b)_n\cnd n, \KP(n))=O(1)$, then (bound for the complexity of a pair) $\KP((a)_n,(b)_n\cnd n, \KP(n))=O(1)$, so 
\bc $\KP(a_0b_0a_1b_1\ldots a_{n-1}b_{n-1}\cnd n,\KP(n))=O(1)$, \ec and therefore $\KP(a_0b_0a_1b_1\ldots a_{n-1}b_{n-1})=\KP(n)+O(1)$. It remains to note that $\KP(n)=\KP(2n)$ and that we can extend the equality to sequences of odd length, since adding one bit changes the complexity of the sequence and its length only by $O(1)$. The similar result for low sequences is not obvious: if each of the sequences $a$ and $b$ separately do not change the complexity function, why their join is also powerless in this regard? Not obvious at all.

The proof of the equivalence (every $\KPt$-trivial sequence is $\KPt$-low) requires a rather complicated combinatorial construction. It may be easier to start with a weaker statement and show that every $\KPt$-trivial sequence is weaker (as an oracle) than the halting problem. (It follows from the equivalence statement, since $\mathbf{0}'$ is not $\KPt$-trivial, see above.) This argument is given in the next section. On the other hand, the full proof is not so complicated, so the reader may also skip the next section and to read the equivalence proof without this training.

\section{$\KPt$-trivial sequence cannot compute $\mathbf{0}'$}

In this section we prove that a $\KPt$-trivial sequence cannot compute $\mathbf{0}'$, in the following equivalent version:

\begin{theorem}
No $\KPt$-trivial sequence can compute all enumerable sets.
\end{theorem}

(Note that together with the existence result proved above it gives a solution of the Post problem, the existence of non-complete enumerable undecidable sets.)

\begin{proof}
The proof consists of several steps.

\subsubsection*{Game reformulation.}
We use the game argument and consider the following game with full information. The opponent approximates some sequence $A$ by changing the values of Boolean variables $a_0,a_1,a_2,\ldots$ (so $A(i)$ is the limit value of $a_i$). He also assigns increasing weights to strings; the total weight should not exceed $1$. We assume that initially all weights are zeros (and that the initial values of $a_i$ are also zeros---just to be specific).

We assign increasing weights to integers (lengths); the sum of our weights is also bounded by $1$. Also we construct some set $W$ by (irreversibly) adding elements to it.

The opponent wins the game if
\begin{itemize}
\item each variable $a_i$ is changed only finitely many times (so some limit sequence $A$ appears);
\item the (opponent's) limit weight of $(A)_i$, the $i$-bit prefix of $A$, is greater than (ours) limit weight of $i$, up to some multiplicative constant;
\item the set $W$ is Turing-reducible to $A$.
\end{itemize}

It is enough to construct a computable winning strategy in this game. Indeed, assume that some $\KPt$-trivial $A$ computes $\mathbf{0}'$. We know that then $A$ is limit computable, so the opponent can computably approximate it, and at the same time approximate from below the a priori probabilities for all strings (ignoring our moves, as usual). Our strategy will them behave computably, generating some lower semicomputable semimeasure on lengths,  and some enumerable set $W$. Then, according to the definition of the game, either this semimeasure is not matched by $\m((A)_i)$, or $W$ is not Turing-reducible to $A$. In the first case $A$ is not $\KPt$-trivial; in the second case $A$ is not Turing-complete.

\subsubsection*{Reduction to a game with fixed machine and constant.}
How can we win computably this game? First we consider a simpler game when the opponent has to declare in advance the constant $c$ that relates the two semimeasures, and the machine $\Gamma$ that reduces $W$ to $A$. Imagine that we can win this game: for each $c$ and $\Gamma$ we have a (uniformly) computable strategy that wins in the $c$-$\Gamma$-game (defined in a natural way). Since the constant $c$ in the definition of $c$-$\Gamma$-game is arbitrary, we may use $c^2$ instead of $c$ and assume by scaling that we can force the opponent to spend more than $1$ while using only $1/c$ total weight and allowing him to match our moves up to factor $c$.

Now we mix the strategies for different $c$ and $\Gamma$ into one strategy. Note that two strategies that simultaneously increase weights of some lengths, can only help each other, so we need only to ensure that the sum of the increases made by all strategies is bounded by~$1$.  More care is needed for the other part: each of the strategies constructs its own $W$, so we should isolate them. For example, to mix two strategies, we split $\mathbb{N}$ into two parts $N_1$ and $N_2$, say, odd and even numbers, and let the first/second strategy construct a subset of $N_1$/$N_2$. Of course, each strategy should then beat not the machine $\Gamma$ itself, but its restriction on $N_i$ (the composition of $\Gamma$ and the embedding of $N_i$ into $\mathbb{N}$). In a similar way we can mix countably many strategies (splitting $\mathbb{N}$ into countably many countable sets in a computable way). 

It remains to consider some computable sequence $c_i\to\infty$ such that $\sum 1/c_i \le 1$ and a computable sequence $\Gamma_i$ that includes every machine $\Gamma$ infinitely many times. (The latter is needed because we want every $\Gamma$ to be beaten with arbitrarily large constant $c$.) Combining the strategies for these games as described, we get a computable winning strategy for the full game.

It is convenient to scale the $c$-$\Gamma$-game and require the opponent to match our weights exactly (without any factor) but allow him to use total weight $c$ (instead of~$1$).  We will prove the existence of the winning strategy by induction (the strategy for some $c$ will be used to construct a strategy for larger $c$). To start, let us first construct the strategy for $c<2$. 

\subsubsection*{Winning a game with $c<2$: strong strings}

This winning strategy deals with some fixed machine $\Gamma$ and ensures $\Gamma^A\ne W$ at one fixed point (say, $0$); the other points are not used. Informally, we wait until the opponent puts a lot of weight on strings that imply $0\notin \Gamma^A$. If this never happens, we win in one way; if this happens, we then add $0$ to $W$ and win in another way.

Let us explain this more formally. We say that a string $u$ is \emph{strong} if it (as a prefix of $A$) forces $\Gamma^A(0)$ to be equal to $0$ (i.e., $\Gamma$ outputs $0$ on input $0$ using only oracle queries in $u$). Simulating the behavior of $\Gamma$ for different oracles, we can enumerate all strong strings. During the game we look at the following quantity:
\begin{quote}
\emph{the total weight that our opponent has put on all (known) strong strings}.
\end{quote}
This quantity may increase because the opponent puts more weight or because we discover new strong strings, but never decreases. We try to make the opponent to increase this quantity (see below about how it is done). As we shall see, if he refuses, he loses the game (and the element $0$ remains outside $W$). If this quantity becomes close to $1$, we change our mind and add $1$ into $W$. After that all the weight put on strong strings is lost: they cannot be the prefixes of $A$ such that $\Gamma^A=W$, if $1\in W$. So we can make our total weight equal to $1$ in an arbitrary way (adding weight somewhere if the total weight was smaller than~$1$), and the opponent needs to use additional weight $1$ along some final $A$ that avoids all strong strings, therefore his weight becomes close to $2$.

\subsubsection*{Winning a game with $c=1$: gradual increase}

So our goal is to force the opponent to increase the total weight of strong strings.  Let us describe our strategy as a set of substrategies (processes) that run in parallel. For each strong vertex $x$ there is a process $P_x$; we start it when we discover that $x$ is strong. This process tries to force the opponent to increase the weight of some vertex above $x$ (this vertex is automatically strong); we want to make the lengths of these strings different, so for every string $x$ we fix some number $l_x$ that is greater than $|x|$, the length of $x$.

Process $P_x$ is activated when the current path goes through vertex $x$ and sleeps otherwise (for example, it may never become active). So---when running---it always sees that the current path goes through $x$. The process gradually increases the weight of length $l_x$: it adds some small $\delta_x$ to the weight of $l_x$ and waits until the opponent matches\footnote{A technical remark: note that in our description of the game we have required that the opponent's weight along the path is strictly greater than our weight: if this is true in the limit, it happens on some finite stage.} this weight along the current path (whatever this path is), then increases the weight again by $\delta_x$, etc. The values of $\delta_x$ are fixed for each $x$ in such a way that $\sum_x \delta_x$ is small.  The process repeats this increase until it gets a termination signal from the supervisor (see below when this happens). Note that at any moment the current path can change in such a way that $x$ is no more on it; then $P_x$ is put to sleep until $x$ becomes on the current path again (and this may never happen).

The supervisor sends the termination signal to all the processes when (and if) the total weight they have used comes close to $1$. After that the strategy adds $0$ to $W$, as explained above. 

Let us show that it is indeed the winning strategy. Consider a game where it is used. By construction, we do not violate the weight restriction. If the opponent has no limit path, he loses. If there is a limit path, but it does not has strong vertices on it, he also loses ($\Gamma^A(0)\ne 0$ for the limit $A$ while $0\notin W$). If $x$ is a strong vertex on the limit path, the process $P_x$ is active starting from some moment, so it can use all the weight (unless other processes do it earlier); in both cases the termination signal is sent. It remains to note that at this moment (when the termination signal is sent) the total weight put on strong string is close to $1$: indeed, all the weight put by us is matched by the opponent, except for one portion of size $\delta_x$ for each vertex $x$; the used weight is close to $1$, and only small amount ($\sum_x\delta_x$) can be lost.

This argument shows how we can win the $c$-$\Gamma$-game for $c<2$.

\subsubsection*{Induction statement}

The idea of the induction step is simple: instead of forcing the weight increase for (some extension of) a strong vertex $u$ directly, we can recursively call the described strategy at the subtree rooted at $u$ (adding or not adding some other element to $W$ instead of $0$), and save our weight (almost twice, since we know how to win the game with $c$ close to~$2$). In this way we can win the game for arbitrary $c<3$, and so on.

To be more formal, we consider a recursively defined process  
$P(k,x,\alpha, L, M)$ with the following parameters:
\begin{itemize}
\item $k>0$ is a rational number, the required coefficient of weight increase;
\item $x$ is the root of the subtree where the process operates;
\item $\alpha>0$ is also a rational number, our ``budget'' (how much weight we are allowed to use);
\item $L$ is an infinite set of integers (lengths where our process may increase weight);\footnote{To use infinite sets as parameters, we should restrict ourselves to some class of infinite sets. For example, we may consider infinite decidable sets and represent them by programs enumerating their elements in the increasing order.}
\item $M$ is and infinite set of integers (numbers that our process may add to $W$).
\end{itemize}

The process can be started (or waked up) only when $x$ is a prefix of the current path $A$, and is put to sleep when $A$ changes and this is no more true, until $x$ becomes the path of the current path again. It is guaranteed that $P$ never violates the rules (about $\alpha$, $L$, and $M$). Running in parallel with other processes (as a part of the game) and assuming that other processes do not touch lengths in $L$ and numbers in $M$, the process $P(k,x,\alpha,L,M)$ guarantees (if not put into sleep forever or terminated externally) that one of the following possibilities happens:
\begin{itemize}
\item either limit path $A$ does not exist,
\item or $W\ne \Gamma^A$ for limit $A$,
\item or the opponent never matches some weight put on some length in $L$,
\item or the opponent spends more than $k\alpha$ weight on vertices above $x$ with lengths in $L$.
\end{itemize}
    
\subsubsection*{Induction base}     
     
Now we can adapt the construction of the previous section and construct a process $P(k,x,\alpha,L,M)$ with these properties for arbitrary $k<2$. For each $y$ above $x$ we select some $l_y\in L$ (different for different $y$), and select some positive $\delta_y$ such that $\sum\delta_y$ is small compared to the budget $\alpha$. We choose some $m\in M$ and consider $y$ (a vertex above $x$) as strong if it guarantees that $\Gamma^A(m)=0$. Then for all strong $y$ we start the process $P_y$ that  is activated when $y$ is in the current path and increases the weight of $l_y$ in $\delta_y$-steps waiting until the opponent matches it. We terminate all the processes when (and if) the total weight used becomes close to $\alpha$, and then add $m$ to $W$ (making all the weight placed by the opponent on strong vertices useless).

The restrictions are satisfied by the construction. Let us check the promised property. If there is no limit path, there is nothing to check. If the limit path does not go through $x$, we have no obligations (the process is put into sleep forever). Assume that the limit path goes through $x$ and the process was not terminated externally. If there is no strong vertex on the limit path $A$, then $m\notin W$, but $\Gamma^A(m)$ is not $0$, so $W\ne\Gamma^A$. If there is a strong vertex $y$ on the limit path, then the process $P_y$ was started and worked without interruptions (starting from some moment), so either some of the $\delta_y$-increases was not matched (third possibility) or the termination signal was sent. In the latter case the total weight used was close to $\alpha$, and after adding $m$ to $W$ it is lost, so either our weight is not matched or almost $2\alpha$ is spent by the opponent above $x$ (recall that all processes are active only when the current path goes through~$x$). 

\subsubsection*{Induction step}

The induction step is similar: we construct the process $P(k,x,\alpha,L,M)$ in the same way as for the induction base. The difference is that instead of $\delta_y$-increase in the weight of $l_y$ the process $P_y$ now recursively calls $$P(k',y,\delta_y\,,L',M')$$ for vertex $y$ with some smaller $k'$ (one can take $k'=k-0.5$), the budget $\delta_y$ and some $L'\subset L$ and $M'\subset M$. If the started process forces the opponent to spend more than $k'\delta_y$ on the vertices above $y$ with lengths in $L'$, then a new process $$P(k',y,\delta_y\,,L'',M'')$$ is started for some other $L''\subset L$ and $M''\subset M$, etc. All the subsets $L',L'',\ldots$ should be disjoint (and also disjoint for different $y$), as well as $M',M'',\ldots$. So we should first of all split $L$ into a sum of disjoint infinite subsets $L_y$ parametrized by $y$ and then split each $L_y$ into $L_y'+L_y''+\ldots$ (for the first, second, etc. recursive calls). The same is done for $M$, but here we in addition to the sets $M_y$ select some $m$ outside all $M_y$. This $m$ is used to define strong vertices as those that guarantee $\Gamma^A(m)=0$. 

We start processes $P_y$ (as described above: each of them makes a potentially infinite sequence of recursive calls with the same $k'=k-0.5$ and budget $\delta_y$) for all discovered strong vertices $y$ (putting them into sleep when $y$ is not on the current path). We take note of the total weight used by all $P_y$ (for all $y$) and sent a termination signal to all $P_y$ when this weight comes close to the threshold $\alpha$ (so it never crosses this threshold).

Let us show that we achieve the declared goal (assuming that the recursive calls achieve it).  First, the restrictions about $L$, $M$ and $\alpha$ are guaranteed by the construction. If there is no limit path, we have no other obligations. If the limit path exists but does not go through $x$, our process will be externally put to sleep, and again we have no obligations. So we may assume that the limit path goes through $x$ and that our process is not  terminated externally. If the limit path does not go through any strong vertex (defined using $m$), then $W\ne \Gamma^A$ for the limit path $A$, since $m\notin W$ and $\Gamma^A(m)$ does not output $0$. If the limit path goes through some strong $y$, the process $P_y$ will be active starting from some point, and makes recursive calls $P(k',y,\delta_y\,,L',M')$,\ldots, $P(k',y,\delta_y\,,L'',M'')$, etc. Now we use the inductive assumption and assume that these calls achieve their declared goals. Consider the first call. If it succeeds by achieving one of three first alternatives, then we are done. If it succeeds by achieving the fourth alternative, i.e., by spending more than $k'\delta_y$ on the weights from $L'$, then the second call is made, and again either we are done or the opponent spends more than $k'\delta_y$ on the weights from $L''$. And so on: at some point we either succeed globally, or exhaust the budget and our main process sends the termination signal to all $P_y$. Then all the weight spent, except for the $\delta_y$'s for the last call at each vertex, is matched by the opponent with factor $k'$, and on the final path the opponent has to match it with factor $1$, so we are done (assuming that $k<k'+1$ and $\sum \delta_x$ is small enough).

This finishes the induction step, so we can win every $c$-$\Gamma$-game by calling the recursive process at the root. We have proven that  $\KPt$-trivial sets do not compute $\mathbf{0}'$.
\end{proof}

\section{$\KPt$-trivial sequences are $\KPt$-low}

Now we want to prove the promised stronger result:

\begin{theorem}
All $\KPt$-trivial sequences are $\KPt$-low.
\end{theorem}

\begin{proof}
Let us start with some preparations.

First, we already know that $\KPt$-trivial sequences are $\mathbf{0}'$-computable. So it is enough to show that every $\KPt$-trivial $\mathbf{0}'$-computable sequence is $\KPt$-low.

Second, let us provide a convenient way to represent the a priori probability $\m^A(\cdot)$ with oracle $A$. (We need to show that it is not significantly  larger than $\m(\cdot)$.) We may describe $\m^A$ in the following way. The sequence $A$ is a path in a full binary tree. Imagine that at every vertex of a tree there is a label of the form $(i,\eta)$ where $i$ is an integer, and $\eta$ is a non-negative rational number. This label is read as ``please add $\eta$ to the weight of $i$''. We assume that the labeling is computable. We also require that for every path in the tree the sum of all rational numbers along the path does not exceed $1$. Having such a labeling, and a path $A$, we can obey all the labels along the path, and get a semimeasure on integers. This semimeasure is semicomputable with oracle $A$.

 In fact, this construction is quite general: for every machine $M$ with oracle generating a lower semicomputable discrete semimeasure, we can find a labeling that gives the same semimeasure (in the way described) for every oracle.   Indeed, we may simulate the behavior of $M$ for different oracles $A$, and look which part of $A$ was read when some increase in the output semimeasure happens. This can be used to create a label at some tree vertex. We need to make the labeling computable; also, according to our assumption, each vertex adds weight only to one object, but both requirements can be easily fulfilled by postponing the weight increase (we push the queue of postponed requests up the tree). If the sum of the increase requests along some path becomes greater than $1$, this means that for $A$ extending this path we do not get a semimeasure. As usual, we can trim the requests and guarantee that we get semimeasures along all paths, not changing the existing semimeasures. 

We may assume now that some computable labeling is fixed such that for every path $A$ the resulting semimeasure (obtained by fulfilling all the requests along the path) is $\m^A$.

\subsubsection*{Game description}

We prove this theorem by providing a winning strategy in some game. (If you have read the previous section, note that the part related to $\KPt$-triviality is the same.) The opponent approximates some sequence $A$ by changing the values of Boolean variables $a_0,a_1,a_2,\ldots$ and assigns increasing weights to strings; the total weight should not exceed $1$. We assume that initially all weights are zeros (and that the initial values of $a_i$ are also zeros).

We assign increasing weights to integers (\emph{lengths}); the sum of our weights is also bounded by $1$. This semimeasure will be compared to the opponent's weights along his path. We also assign increasing weights to another type of integers, called \emph{objects}: on them we compare our semimeasure with the semimeasure $\m^A$ (determined by the opponent's limit path $A$)\footnote{Formally speaking, we construct two semimeasures on integers; to avoid confusion, it is convenient to call their arguments ``lengths'' and ``objects''.}.

\medskip

The opponent wins the game if the following three conditions are satisfied:
\begin{itemize}
\item the limit sequence $A$ exists;
\item opponent's semimeasure along the path \hbox{$*$-exceeds} our semimeasure on lengths, i.e., there exists some $c>0$ such that for all $i$ the opponent's weight of $(A)_i$ is greater than (our weight of $i$)/$c$;
\item our semimeasure on objects does not \hbox{$*$-exceed} $\m^A$.
\end{itemize}

It is enough to construct a computable winning strategy in this game. Indeed, then we can play it against $\mathbf{0}'$-computable sequence $A$ and the universal semimeasure on strings. Then our moves are computable, and we generate semimeasures on lengths and objects. The winning condition guarantees now that (assuming the limit path $A$ exists) either the opponent's semimeasure along the path is not maximal (so $A$ is not $\KPt$-trivial), or that $\m^A$ is $*$-bounded by $\m$ (so $A$ is $\KPt$-low).

As in the previous theorem, we consider an easier (for us) version of the game where the opponent starts the game by declaring some constant $c$ that he plans to achieve (in the second condition), and we need to beat only this $c$. If we can win this game with $c=2^{2k}$, then we can win the game with $c=2^k$ using only $2^{-k}$ total weight, so we can combine these strategies for all $k$ (we also assume that the total weight on objects for $k$th strategy is also bounded by $2^{-k}$, but this is for free, since we need only to $*$-exceed $\m^A$ without any restrictions on the constant). So it remains to win the game for each $c$. It is convenient to scale that game and assume that the opponent needs to match our weights on length exactly (not up to $1/c$-factor), but his total weight is bounded by $c$ (not $1$).

\subsubsection*{Winning the game for $c<2$}

For $c=1$ the game is trivial, since we require that the opponent's weight along the path is strictly greater than our weight on lengths, so it is enough to assign weight $1$ to some length. We start our proof by explaining the strategy for the case $c<2$.

The idea can be explained as follows.  The na\"\i ve strategy is to believe all the time that the current path $A$ is final, and just assign the weights to objects according to $\m^A$, computed based on the current path $A$. (In fact, at each moment we look at some finite prefix of $A$ and follow the labels that appear on this prefix.)  If indeed $A$ never changes, this is a valid strategy: we achieve $\m^A$, and never exceed the total weight $1$ due to the assumption about the labels. But if the path suddenly changes, then all the weight placed because of vertices on old path outside the new path, is lost. If we now follow all the labels on the new path, then our total weight  on objects may exceed $1$ (it was bounded only along every path, and now we have a part of the old path plus the new path).

There is some partial remedy: we can match the weights only up to some constant (say, use only $1\%$ of what the label asks). This is enough since the game allows us to match the measure with arbitrary constant factor. In this way we can tolerate up to hundred changes in the path (each new path generates new weight at most $0.01$), but this does not really help, since the number of changes (of course) is not bounded. Not a surprise, since we did not use the other part of the game, and without this part there is no hope (not all $\mathbf{0}'$-computable $A$ are low).

How can we discourage the opponent to change the path? We can assign a non-zero weight to some length and wait until the opponent matches this weight along the current path. This is a kind of incentive for the opponent not to leave the vertex where he has put some weight: if the opponent leaves it, he would be forced to waste this weight (and put the same weight along the final path for the second time). After that we could implement a request (label) at this vertex: at least we know that if later the path changes and we lose some weight (by implementing the request not on the limit path), the opponent loses some weight, too. This helps if we are careful enough.

Let us explain the details. It would be convenient to represent the strategy as the set of parallel processes: for each vertex $x$ we have a process  $P_x$ that is awake when $x$ is a prefix of the current path (and sleeps when $x$ is not). When awake, the process $P_x$ tries to create the incentive for the opponent not to leave $x$, by forcing him to increase the weight of some vertex above $x$. To make the processes more independent (and the analysis simpler), let us assume that for every vertex $x$ some length $l_x\ge |x|$ is chosen, lengths assigned to different vertices are different, and $P_x$ increases only the weight of $l_x$.

Now we are ready to describe the process~$P_x$. Assume that $x$ carries the request $(i,\eta)$ that asks to add $\eta$ to the weight of object $i$. The process $P_x$ increases the weight of $l_x$ (adding small portions to it and waiting after each portion until the opponent matches this increase along the current path, in some vertex above $x$).  When (and if) the weight of $l_x$ reaches $\varepsilon\eta$ (where $\varepsilon$ is some small positive constant; the choice of $\varepsilon$ depends on $c$, see below), the process increases the weight of object $i$ also by $\varepsilon\eta$ and terminates. 

The processes $P_x$ for different vertices $x$ run in parallel independently, except for one thing: if the total weight spent by all the processes reaches $1$, we terminate all of them (so the increase that would make the total weight greater than $1$ is blocked); after that our strategy stops working (in the hope that the opponent would be unable to match already existing weights not crossing the threshold $c$).

About the small portions: for each vertex $x$ we choose in advance the size $\delta_x$ of the portions used by $P_x$, in such a way that $\sum_x \delta_x <\varepsilon$. (We use here the same small $\varepsilon$ as above.) In this way we guarantee that the total loss (last portions that were not matched because the opponent changes the path instead and does not return, so the process is not resumed) is bounded by $\varepsilon$.

It remains to prove that this strategy wins $c$-game for $c$ close to $2$, assuming that $\varepsilon$ is small enough. First note two properties that are true by construction:
\begin{itemize}
\item the sum of our weights for all lengths does not exceed $1$;
\item at every moment the sum of (our) weights for all objects does not exceed by the sum of (our) weights for all lengths.
\end{itemize}
(Indeed, we stop the strategy preventing the violation of the first requirement, and the second is guaranteed for each $x$-process and therefore for the entire strategy.)

If there is no limit path, the strategy wins the game by definition. So assume that limit path $A$ exists. Now we count separately the weights used by processes $P_x$ for $x$ is on the limit path $A$, and for $x$ outside $A$. Since the weights for $x$ are limited by the request in $x$ (with factor $\varepsilon$) and sum of the all requests along $A$ is at most $1$, the sum of the weights along $A$ is bounded by $\varepsilon$. Now there are two possibilities: either the strategy was stopped (when trying to cross the threshold) or it runs indefinitely. 

In the first case the total weight is close to $1$ (at least $1-\varepsilon$, since the next increase will cross $1$, and all the portions $\delta_x$ are less than $\varepsilon$). So the weight used by processes outside $A$ is at least $1-2\varepsilon$, and if we do not count the last (unmatched) portions, we get at least $1-3\varepsilon$ of weight that the opponent needs to match twice: it was matched above $x$ for $P_x$, and then should be matched again along the limit path (that does not go through $x$; recall that we consider the vertices outside the limit path). So the opponent needs to spend at least $2-6\varepsilon$, otherwise he loses.

In the second case each process $P_x$ for $x$ on the limit path is awake starting from some moment, and is never stopped, so it reaches its target value $\varepsilon\eta$ and adds $\varepsilon\eta$ to the object $i$ (here $(i,\varepsilon)$ is the request in vertex $x$). So our weights on the objects match $\m^A$ for limit path $A$ up to factor $\varepsilon$, and the opponent loses. (We know also that the total weight on objects does not exceed $1$, since it is bounded by the total weight on lengths.)

We therefore have constructed a winning strategy for $2-6\varepsilon$ game, and by choosing a small $\varepsilon$ can win $c$-game for every $c<2$.

\subsubsection*{Using this strategy on a subtree}

Preparing for the recursion, let us look at the described strategy and modify it for using in a subtree rooted at some vertex $x$. We also scale the game and assume that we have some budget $\alpha$ that we are allowed to use (instead of total weight $1$). To guarantee that the strategy does not interfere with other actions outside $x$-subtree we agree that it uses lengths only from some infinite set $L$ of length and nobody else touches these lengths. Then we can assign $l_y\in L$ for every $y$ in the subtree and use them as before. 

Let us describe the strategy in more details. It is composed of processes $P_y$ for all $y$ above $x$. When $x$ is not on the actual path, all these processes sleep, and the strategy is sleeping. But when path goes through $x$, some processes $P_y$ (for $y$ on the path) become active and start increasing the weight of length $l_y$ by small portions $\delta_y$ (the sum of all $\delta_y$ now is bounded by $\alpha\varepsilon$, since we scaled everything by $\alpha$). A supervisor controls the total weight used by all $P_y$, and as soon as it reaches $\alpha$, terminates all $P_y$. When the process $P_y$ reaches the weight $\alpha\varepsilon\eta$, it increases the weight of object $i$ by $\alpha\varepsilon\eta$ (here $(i,\eta)$ is the request at vertex $y$). So everything is as before, but with factor $\alpha$ and only on $x$-subtree.

What does this strategy guarantee?
\begin{itemize}
\item The total weight on lengths used by it is at most $\alpha$.
\item The total weight on objects does not exceed the total weight on lengths.
\item If the limit path  $A$ exists and goes through $x$, then either
   \begin{itemize}
   \item the strategy halts and the opponent either fails to match all the weights or spends more than $c\alpha$ on $x$-subtree; or
   \item the strategy does not halt, and the semimeasure on objects generated by this strategy, $*$-ex\-ceeds $\m^A$, if we omit from $\m^A$ all the requests on the path to $x$.
   \end{itemize}
\end{itemize}
The argument is the same: if the limit path exists and the strategy does not halt, then all the requests along the limit path (except for finitely many of them below $x$) are fulfilled with coefficient $\alpha\varepsilon$. If the strategy halts, the weight used along the limit path does not exceed $\alpha\varepsilon$ (since the sum of requests along each path is bounded by $1$), the weight used in the other vertices of $x$-subtree is at least $\alpha(1-2\varepsilon)$, including at least $\alpha(1-3\varepsilon)$ matched weight that should be doubled along the limit path, and we achieve the desired goal for $c=2-6\varepsilon$.

\textit{Remark}. In the statement above we have to change $\m^A$ by deleting the requests on the path to $x$. We can change the construction by moving requests up the tree when processing vertex $x$ to get rid of this problem. One may also note that omitted requests deal only with finitely many objects, so one can average the resulting semimeasure with some semimeasure that is positive everywhere. So we may ignore this problem in the sequel.

\subsubsection*{How to win the game for $c<3$}

Now we make the crucial step: show how we can increase $c$ by recursively using our strategies. Recall our strategy for $c<2$, and change it in the following ways:
\begin{itemize}
\item Instead of assigning some length $l_x$ for each vertex $x$, let us assign an infinite (decidable uniformly in $x$) set $L_x$ of integers; all elements should be greater than $|x|$ and for different $x$ these sets should be disjoint (it is easy to achieve this).
\item We agree that process $P_x$ (to be defined) uses only lengths from $L_x$.
\item Again $P_x$ is active when $x$ is on the current path, and sleeps otherwise.
\item Previously $P_x$ increased the weight of $l_x$ in small portions, and after each small increase waited until the opponent matched this increase along the current path. Now, instead of that, $P_x$ calls the $x$-strategy described in the previous section, with small $\alpha=\delta_x$, waits until this strategy terminates forcing the opponent to spend almost $2\delta_x$, then calls another instance of $x$-strategy, waits until it terminates, and so on. For that, $P_x$ divides $L_x$ into infinite subsets $L_x^1+L_x^2+\ldots$, using $L_x^s$ for $s$th call of $x$-strategy, and using $\delta_x$ as the budget for each call. 
\end{itemize}

There are several possibilities for the behavior of $x$-strategy called recursively. It may happen that it does infinitely many steps. This happens when $x$ is a part of the limit path,  $x$-strategy never exceeds its budget $\delta_x$, and the entire strategy does not come close to $1$ in its total spending. It this case we win the game (the part of the semimeasure on objects due to $x$-strategy is enough to $*$-exceed $\m^A$).\footnote{This case is called ``the golden run'' in the original exposition of the proof.}

If $x$ is not in the limit path, the execution of $x$-strategy may be interrupted; in this case we know only that it spent not more than its budget, and that the weight used for objects does not exceed the weight used for lengths. (This is similar to the case when an increase at $l_x$ was not matched because the path changed.)

The $x$-strategy may also terminate. In this case we know that the opponent used almost twice the budget ($\delta_x$) on the extensions of $x$, and a new call of $x$-strategy is made for another set of lengths. (This is similar to the case when the increase at $l_x$ was matched; the advantage is that now the opponent used almost twice our weight!)

Finally, the strategy may be interrupted because the total weight used by $x$-processes for all $x$ came close to $1$. Similar thing happened for the case $c<2$. After that everything stops, and we just wait until the opponent will be unable to match all the existing weights or forced to use total weight close to $3$. Indeed, most of our weight (except for $O(\varepsilon)$) was used not on the limit path and already matched with factor close to $2$ there --- so matching it again on the limit path makes the total weight close to $3$. 

\subsubsection*{Induction step}

Now it is clear how one can continue this reasoning and construct a winning strategy for arbitrary $c$. To get a strategy for some $c$,  we follow the described scheme, and the process $P_x$ makes sequential recursive calls of $c'$-strategies for smaller $c'$ (the difference should be less than $1$, for example, we can use $c'=c-0.5$). More formally, we recursively define a process $S(c, x,\alpha,L)$ where $c$ is the desired amplification, $x$ is a vertex, $\alpha$ is a positive rational number (the budget), and $L$ is an infinite set of integers greater than $|x|$.\footnote{The pedantic reader may complain that the parameter is an infinite sets. It in enough to consider infinite sets from some class, say, decidable sets (as we noted in the previous section), or just arithmetic progressions, they are enough for our purposes. Indeed, an arithmetic progression can be split into countably many arithmetic progressions. For example, $1,2,3,4,\ldots$ can be split into $1,3,5,7,\ldots$ (odd numbers), $2,6,10,14\ldots$ (odd numbers times $2$), $4,12,20,28,\ldots$ (odd numbers times $4$), etc.} The requirements for $S(c,x,\alpha,L)$:

\begin{itemize}
\item It increases only weights of lengths in $L$.
\item The total weight used for lengths does not exceed~$\alpha$.
\item At each step the total weight used for objects does not exceed the total weight used for lengths.
\item Assuming that the process is not terminated externally (this means that $x$ belongs to the current path, starting from some moment), it may halt or not, and:
\begin{itemize}
    \item If the process halts, the opponent uses more that $c\alpha$ on strings that have length in $L$ and are above $x$.
    \item If the process does not halt and the limit path $A$ exists, the semimeasure generated on objects by this process only, $*$-exceeds $\m^A$.
\end{itemize}
\end{itemize}

The implementation of $S(c,x,\alpha,L)$ uses recursive calls of $S(c-0.5,y,\beta,L')$; for each $y$ above $x$ a sequence of those calls is made with $\beta=\delta_y$ and $L'$ that are disjoint subsets of $L$ (for different $y$ these $L'$ are also disjoint), similar to what we have described above for the case $c<3$. 
\end{proof}

\section{$\KPt$-low and ML-low oracles}

There is one more description of $\KPt$-low (or $\KPt$-trivial) sequences: they are sequences that (being used as oracles) do not change the notion of Martin-L\"of randomness. As almost all notions of the computability theory, the notion of Martin-L\"of randomness can be relativized by an oracle $a$ (this means that the algorithm that enumerates covers for an effectively null set, now may use oracle $a$). In this way we get (in general) more effectively null sets, and therefore less random sequences. However, for some $a$ the class of Martin-L\"of random sequences remains unchanged.

\begin{definition}
A sequence $a$ is $\MLR$-low if every Martin-L\"of random sequence remains Martin-L\"of random with oracle $a$.
\end{definition}

The Schnorr--Levin criterion of randomness in terms of prefix complexity shows that if $a$ is $\KPt$-low, then $a$ is also $\MLR$-low. The other implication is also true, but more difficult.

\begin{theorem}
Every $\MLR$-low sequence is $\KPt$-low.
\end{theorem} 

We will prove a more general result. 

\begin{definition}
Let $a$ and $b$ be two sequences (considered as oracles). We say that $a\leK b$ if $$\KP^b(x)\le \KP^a(x)+O(1).$$ We say that $a\leMLR b$ if every Martin-L\"of random with oracle $b$ sequence is also Martin-L\"of random with oracle $a$. \end{definition}

If an oracle $b$ is stronger in Turing sense (=computes) some oracle $a$, then $b$ gives larger effectively null sets, smaller set of random sequences, and smaller complexity function, so we have $a\leMLR b$ and $a\leK b$. So both orderings are more coarse than the Turing degree ordering.

Using this definitions, we can reformulate the definitions: sequence $a$ is $\KPt$-low if $a\leK \mathbf{0}$ and is $\MLR$-low if $a\leMLR \mathbf{0}$. So it is enough to prove the following result:

\begin{theorem}
$a\leK b \ \Leftrightarrow a\leMLR b$ for every $a$ and $b$.
\end{theorem}

\begin{proof}
Again in one direction (from left to right) it follows directly from the Schnorr--Levin randomness criterion. The other direction is more difficult\footnote{This was to  be expected. The relation $a\leK b$ is quantitative: two functions coincide with $O(1)$-precision; the relation $a\leMLR b$ is more qualitative, we speak here about yes/no question. One can also consider the quantitative version, with randomness deficiencies, but this is not needed: the relation $\leMLR$ is strong enough.}, and will be split in several steps.

Recall that the set of non-random (in Martin-L\"of sense; we do not use other notions of randomness here) sequences can be described using universal Martin-L\"of test, as the intersection of effectively open sets
     $$
 U_1\supset U_2 \supset U_3\supset\ldots     
     $$
where $U_i$ has measure at most $2^{-i}$. The following observation goes back to Ku\v cera and says that  the first layer of this test,  the set $U_1$, is enough to characterize all non-random sequences. 
\begin{lemma}\label{kucera}
Let $U$ be an effectively open set of measure less than $1$ that contains all non-random sequences. Then a sequence $a=a_0 a_1 a_2\ldots$ is non-random if and only if all its tails $a_k a_{k+1} a_{k+2}\ldots$ belong to $U$. 
\end{lemma}

\begin{proof}[Proof of Lemma~\ref{kucera}]
If $a$ is non-random, then all the tails are non-random and therefore belong to $U$. In the other direction: represent $U$ as a union of disjoint intervals $[u_i]$ (by $[v]$ we denote the set of all sequences that have prefix $v$). Then $\rho=\sum 2^{-|u_i|}$ is less than $1$. If all tails of $a$ belong to $U$, then $a$ starts with some $u_i$, the rest is a tail that starts with some $u_j$, etc., so $a$ can be split into pieces that are among $u_i$. The set of sequences of the form ``some $u_i$, then something'' has measure $\rho$, the set of sequence of the form ``some $u_i$, some $u_j$, then something'' has measure $\rho^2$, etc. These sets are effectively open and their measures $\rho^n$ effectively converge to $0$. So their intersection is an effectively null set and $a$ is non-random. Lemma is proven.
\end{proof}

The argument gives also the following

\textbf{Corollary}: \emph{a sequence $x$ is non-random if there exists an effectively open set $U$ of measure less than $1$ such that all tails of $x$ belong to $U$}.

\smallskip

This argument can be relativized, so randomness with oracle $X$ can be characterized in terms of $X$-effectively open sets of measure less that $1$: \emph{a sequence $x$ is $X$-nonrandom if there exists an $X$-effectively open set $U$ of measure less than $1$ such that all tails of $x$ belong to $U$}. This gives one implication in the following equivalence (we denote here the oracles by capital letter not to mix them with sequences):

\begin{lemma}\label{nonfull}
Let $A$ and $B$ be two oracles. Then $A\leMLR B$ if and only if every $A$-effectively open set of measure less than $1$ can be covered by some $B$-effectively open set of measure less than $1$.
\end{lemma}

\begin{proof}[Proof of Lemma~\ref{nonfull}]
In one direction (if, $\Leftarrow$) it follows for the discussion above: if $a$ is not $A$-random, its tails can be covered by some $A$-effectively open set of measure less than $1$ and therefore by some $B$-effectively open set of measure less than $1$, so $a$ is not $B$-random.

In the other direction: let $U$ be an $A$-effectively open set of measure $1$ that cannot be covered by any $B$-effectively open set of measure less than $1$. The set $U$ is the union of $A$-enumerable sequence of disjoint intervals $[u_1],[u_2],[u_3]$, etc. Consider a set $V$ that is $B$-effectively open, contains all $B$-non-random sequences and has measure less than $1$ (e.g., the first level of universal $B$-Martin-L\"of test). By assumption $U$ is not covered by $V$, so some interval $[u_i]$ of $U$ is not (entirely) covered by $V$.

The set $V$ has the following special property: if it does not contain \emph{all} points of some interval, then it cannot contain \emph{almost all} points of this interval: the non-covered part has some positive measure. Indeed, the non-covered part is a $B$-effectively closed set, and if it has measure zero, it has $B$-effectively measure zero, so all non-covered sequences are $B$-non-random, and therefore should be covered by $V$.

So we found an interval $[u_i]$ in $U$ whose part of positive measure is outside $V$. Then consider the set $V_1=V/u_i$, i.e., the set of infinite sequences $\alpha$ such that $u_i\alpha\in V$. This is a $B$-effectively open set of measure less than $1$, so it does not cover $U$ (again by our assumption). So there exists some interval $[u_j]$ not covered by $V/u_i$. This means that $[u_i u_j]$ is not covered by $V$. Then we repeat the argument and conclude that non-covered part has positive measure, so $V/u_i u_j$ is a $B$-effectively open set of measure less than $1$, so it does not cover some $[u_k]$, etc.  In the limit we get a sequence $u_i u_j u_k\ldots$ whose prefixes define intervals not covered fully by $V$. Since $V$ is open, this sequence does not belong to $V$, so it is $B$-random. On the other hand, it is not $A$-random, as the argument from the proof of Lemma~\ref{kucera} shows. 
\end{proof}

Let us summarize where we are now. Assuming that $A\leMLR B$, we have shown that every $A$-effectively open set of measure less than $1$ can be covered by some $B$-effectively open set of measure less than $1$. And we need to show somehow that $A\leK B$, i.e., $\KP^B\le \KP^A$ up to an additive constant, or $\m^A\le \m^B$ up to a constant factor. This can be reformulated as follows: \emph{for every lower $A$-semicomputable converging series $\sum a_n$ of reals there exists a converging lower $B$-semicomputable series $\sum b_n$ of reals such that $a_n\le b_n$ for every $n$}.

So to connect our assumption and our goal, we need to convert somehow a converging lower semicomputable series into an effectively open set of measure less than $1$ and vice versa. We may assume without loss of generality that all $a_i$ are strictly less than $1$. Then $\sum a_n<\infty$ is
equivalent to 
    $$
(1-a_0)(1-a_1)(1-a_2)\ldots > 0.    
    $$
This product is a measure of an $A$-effectively closed set 
    $$
 [a_0,1]\times [a_1,1]\times[a_2,1]\times\ldots    
    $$
so its complement 
  $$
 \{ (x_0,x_1,\ldots)\mid (x_0<a_0) \lor (x_1<a_1)\lor \ldots \}   
    $$
is an $A$-effectively open set of measure less than $1$. (Here we split the Cantor space into a countable product of Cantor spaces and identify each of them with $[0,1]$ equipped with standard uniform measure on the unit interval.) We are finally ready to apply our assumption and find some $B$-effectively open set $V$ that contains this complement.
  
Now let us define $b_0$ as supremum of all $z$ such that
 $$[0,z]\times [0,1]\times [0,1]\times\ldots \subset  V$$   
This product is compact for every $z$, and $V$ is $B$-effectively open, so we can $B$-enumerate all rational $z$ with this property, and their supremum $b_0$ is lower $B$-semicomputable. Note that all $z<a_0$ have this property, so $a_0\le b_0$. In a similar way we define all $b_i$ and get a lower $B$-semicomputable series $b_i$ such that $a_i\le b_i$. It remains to show that $\sum b_i$ is finite. Indeed, the set
    $$
 \{ (x_0,x_1,\ldots)\mid (x_0<b_0) \lor (x_1<b_1)\lor \ldots \}   
    $$
is a part of $V$, and therefore has measure less than $1$; its complement 
    $$
 [b_0,1]\times [b_1,1]\times[b_2,1]\times\ldots    
    $$   
has measure $(1-b_0)(1-b_1) (1-b_2)\ldots$, therefore this product is positive and the series $\sum b_i$ converges. This finishes the proof.   
\end{proof}

\subsubsection*{Acknowledgments}
This exposition was finished while one of the authors (A.S.) was invited to National University of Singapore IMS \emph{Algorithmic Randomness} program. The authors thank IMS for the support and for the possibility to discuss several topics (including this exposition) with the other participants. Special thanks to Rupert H\"olzl and Nikolai Vereshchagin. We thank also Andre Nies for the suggestion to write down this exposition for the Logic Blog and comments, and all the participants of the course taught by L.B. at Poncelet lab (CNRS, Moscow), especially Misha Andreev and Gleb Novikov. Last but not least, we are grateful to Joe Miller who presented the proof of equivalence between $\leK$ and $\leMLR$ while visiting LIRMM several years ago.

\bibliographystyle{plain}
\def\cprime{$'$}

\end{document}